\documentclass[10pt,twocolumn,twoside]{IEEEtran}

\usepackage{amssymb,amsfonts,amsmath,amsthm}
\usepackage{mathtools}
\usepackage{eucal,bm}
\usepackage{gensymb,subfigure}
\usepackage{multirow}
\usepackage{booktabs}
\usepackage{times,color,url}
\usepackage{comment}
\usepackage{cite}
\usepackage{graphicx}
\usepackage[all]{xy}
\usepackage{algorithm}
\usepackage{algpseudocode}
\usepackage{subfigure}
\usepackage{bm}
\usepackage{color}
\usepackage{wrapfig}
\usepackage{epsf}
\usepackage{times,mathptm}
\usepackage{url}
\usepackage{nomencl}
\newtheorem{theorem}{Theorem}

\newtheorem{lemma}{Lemma}
\newcommand{\squeezeup}{\vspace{-2.5mm}}

\begin{document}
\title{Structure Learning and Statistical Estimation in Distribution Networks - Part I}
\author{\authorblockN{Deepjyoti~Deka*, Scott~Backhaus\dag, and Michael~Chertkov\ddag\\}
\authorblockA{*Corresponding Author. Electrical \& Computer Engineering, University of Texas at Austin\\
\dag MPA Division, Los Alamos National Lab\\ \ddag Theory Division and the Center for Nonlinear Systems, Los Alamos National Lab}
\thanks{D. Deka is with the Department of Electrical and Computer Engineering, The University of Texas at Austin, Austin, TX 78712. Email: deepjyotideka@utexas.edu}
\thanks{S. Backhaus is with the MPA Division of LANL, Los Alamos, NM 87544. Email: *backhaus@lanl.gov}
\thanks{M. Chertkov is with the Theory Division and the Center for Nonlinear Systems of LANL, Los Alamos, NM 87544. Email: *chertkov@lanl.gov}}
\maketitle

\begin{abstract}
Traditionally power distribution networks are either not observable or only partially observable. This complicates development and implementation of new smart grid technologies, such as those related to demand response, outage detection and management, and improved load-monitoring. In this two part paper, inspired by proliferation of metering technology, we discuss estimation problems in structurally loopy but operationally radial distribution grids from measurements, e.g. voltage data, which are either already available or can be made available with a relatively minor investment. In Part I, the objective is to learn the operational layout of the grid. Part II of this paper presents algorithms that estimate load statistics or line parameters in addition to learning the grid structure. Further, Part II discusses the problem of structure estimation for systems with incomplete measurement sets. Our newly suggested algorithms apply to a wide range of realistic scenarios. The algorithms are also computationally efficient -- polynomial in time -- which is proven theoretically and illustrated computationally on a number of test cases. The technique developed can be applied to detect line failures in real time as well as to understand the scope of possible adversarial attacks on the grid.
\end{abstract}

\section{Introduction}
\label{sec:intro}
The power grid is composed of a network of transmission and distribution lines that enable the transfer of electrical power from generators to loads. The design, operation and control of these networks is typically hierarchical with a major division occurring between the transmission network of high voltage lines connecting sub-stations and power plants, and the distribution network of medium and low voltage lines that connect the transmission sub-stations to the end-users. Here, we focus on distribution networks.

The design of distribution networks may appear to be loopy or meshed, however for practical engineering concerns, the vast majority of distribution grids are operated as "radial" networks, i.e. as a set of non-overlapping trees. Switches in the network are used to achieve one radial configuration out of many possibilities. Each tree in the network has a substation at the root and customers positioned at the other nodes. Switching from one tree-like operational configuration to another is typically caused by system upsets, e.g. faults and outages, and may occur few times a day or even an hour.

The radial configuration distinguishes distribution networks from transmission networks that generally have multiple loops energized all the time to guarantee continuous delivery of power to every node, even in case of occasional line faults and outages. Radial configurations and one-way flow of power have led to much less monitoring, observability, and state estimation in distribution as compared to meshed transmission networks \cite{hoffman2006practical}. The recent proliferation of smart grid technology, including smart meters that measure electricity consumption at the node level, is creating a new opportunities to extract information important to grid operators and planners. Such efforts are also getting additional attention in view of mounting concerns over data security and protection of user privacy \cite{liu2012cyber}.

In this paper (Part I), we seek to \textbf{\textit{develop low-complexity algorithms to learn the current operational structure in `radial' distribution networks using only nodal measurements}}. Nodal measurements may include voltage magnitudes, voltage phase (potentially), and power injections and are typically available at smart meters, pole-mount or pad-mount transformers, and distribution phasor measurement units (PMU). Accurate structural estimation impacts many important applications including failure identification \cite{sharon2012topology}, outage management, and recovery following major and minor disruptions (e.g. hurricanes to individual lightning strikes), grid reconfiguration \cite{baran1989network} for power flow optimization and generation scheduling\cite{hoffman2006practical,lopes2011integration,baran1989network,turitsyn2011options}, and quantifying the need for additional meter placement. From an adversarial viewpoint, our work can be viewed as low-intrusion learning by a rogue agent interested in estimating the grid structure for a data attack \cite{kim2013topology,deka2014attacking}. In the subsequent part (Part II), we will look at developing algorithms that are able to estimate the statistics of power consumption at the grid nodes or estimate the parameters of operational lines in addition to determining the grid's radial structure. Further, we will analyze learning the operational grid structure with missing data, where observations from a subset of nodes are not available.

\subsection{Related Work}
Our work falls in the broad category of `graph learning' problems that have been approached from different directions. For general graphs and graphical models \cite{wainwright2008graphical}, maximum-likelihood structure estimation has been researched in several papers by utilizing prior information such as sparsity of the parameter space \cite{ravikumar2010high,anandkumar2011high}, size of the graph neighborhood \cite{netrapalli2010greedy}, etc. Techniques employed include both traditional convex optimization \cite{ravikumar2010high,ravikumar2011high} as well as greedy learning \cite{netrapalli2010greedy}. For power grids, structure estimation techniques discussed in the literature can be classified based on the type of measurements available as well as assumptions made regarding grid structure and user behavior. In \cite{kekatos2013grid}, a maximum likelihood estimator (MLE) with regularizers for low-rank and sparsity is used to recover the grid structure using locational marginal prices (LMPs). In \cite{he2011dependency}, a model using bus phase angles as a Markov random field for the DC power flow builds a dependency graph based approach to detect faults in grids. In work specific to radial distribution grids, \cite{bolognani2013identification} provides a structure identification algorithm that uses signs within the inverse covariance matrix (or concentration matrix) of voltage measurements to generate a minimum spanning tree. In \cite{sharon2012topology}, topology identification with limited measurements in a distribution grid with Gaussian loads is used to design a machine learning (ML) estimate with approximate schemes. Our work uses ordering of second moments, not a ML approach, to reconstruct a radial grid sequentially from the leaves to the root, making it distinct from previous work. Our algorithm design is based on a linear coupled approximation for lossless AC power flow that is idealized but practical \cite{89BWa,89BWb} for analyzing distribution grids where the line and voltage characteristics limit the accuracy of traditional approximations.
Unlike related work, our topology learning algorithm is agnostic to the load profile distributions or variability in line impedances and requires only a less restrictive assumption on the correlation of load profiles.

The rest of the paper is organized as follows. We start the next section with a description of the distribution grid and summary of the learning problems discussed in the manuscript. Section \ref{sec:pf+corr} and Appendix \ref{app:PF} describe the linear coupled (LC) power flow model and its special case, the DC-resistive power flow model. Statistical trends in observed nodal measurements are discussed in Section \ref{sec:trends}. Next, we use the derived results to design algorithms for learning the distribution grid structure using the power flow models in Section \ref{sec:learn}. Simulation results elucidating the performance of our algorithms on test distribution grids are presented in Section \ref{sec:experiments}. Finally, Section \ref{sec:conclusion} concludes and suggests future directions that will be explored in a subsequent work.

\section{Technical Preliminaries}
\label{tech_intro.tex}
The structure of a radial distribution network has important features that motivates our algorithm development. We discuss the radial structure in detail here and introduce the notation used in this paper. We then formulate the learning problem tackled in this paper in terms of its input data and deliverables and discuss the underlying motivation.

\subsection{Structure of Radial Distribution Network}
\label{sec:structure}
\begin{figure}[!bt]
\centering
\includegraphics[width=0.40\textwidth, height = .30\textwidth]{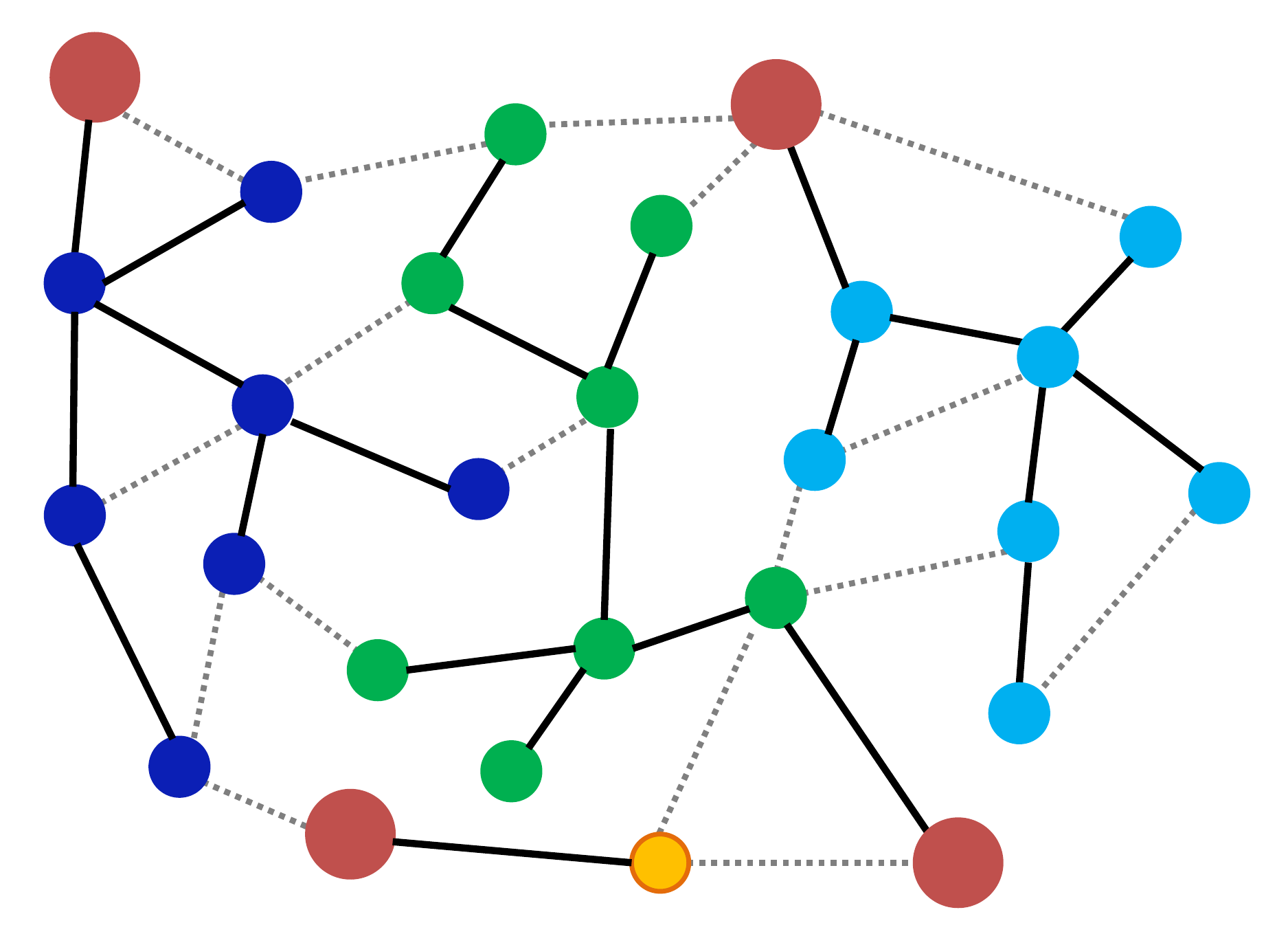}
\caption{Schematic layout of an example of a distribution grid fed from $4$ substations, with substations represented by large red nodes. The operational grid has a base-$4$ spanning forest configuration that is formed by solid lines (black). Dotted grey lines represent open switches. Each tree in the forest has one substation at the root. The load nodes within each tree are marked with the same color.
\label{fig:city}}
\end{figure}

We consider a meshed distribution network which is operated as a union of non-intersecting `radial' trees, i.e. a spanning forest, by configuring switches as shown in Fig.~\ref{fig:city}. There are exponentially many (in the number of switches) possible configurations of spanning forests. The grid-graph with all the switches closed is denoted ${\cal G}=({\cal V},{\cal E})$, where ${\cal V}$ is the set of nodes of the graph and ${\cal E}$ is the set of undirected edges of the graph. We denote nodes with single Roman letter subscripts $a$ and undirected edges with pairs of Roman letter subscripts $(ab)$. The operational grid is a forest denoted by ${\cal F}$ which spans all the nodes in ${\cal V}$. Specifically, ${\cal F}$ is a special subgraph of ${\cal G}$ (${\cal F}\subset {\cal G}$) such that
\begin{itemize}
\item ${\cal F}$ is a union of $K$ non-overlapping trees covering all the nodes of the graph
 \item Each tree contains exactly one of the $K$ bases (substations), ${\cal F}=\cup_{k=1,\cdots,K}{\cal T}_k$.
\end{itemize}
The distribution system ${\cal F}$ is a \textit{\textbf{`base-constrained spanning forest'}} with operational edges ${\cal E}^{\cal F}$ where ${\cal E}^{\cal F}\subset {\cal E}$. Table~\ref{table1} provides other relevant notations (nomenclature) used through out this manuscript to denote various nodal and edge features of the grid $\cal G$ and the operational forest $\cal F$.


\begin{table}[ht]
\caption{Notation Table}
\begin{center}
\begin{tabular}{|l|l|}
\hline
${\cal F}\in{\cal G}$ & \begin{tabular}{|l} a particular forest configuration\\
of the physical distribution network ${\cal G}$ \end{tabular}\\
$\cal V$ & vertex set of ${\cal G}$\\
$N$ & $\#$ of nodes other than sub-stations in $\cal V$ \\
$K$ & number of sub-stations in the network \\
$\cal E$ & edge set of ${\cal G}$\\
${\cal E}^{\cal F}$ & set of edges operational within $\cal F$ \\
${\cal T}_k\in{\cal F}$ & \begin{tabular}{|l} tree within ${\cal F}$
containing\\ the $k^{th}$ sub-station\end{tabular} \\
$M_k$ & reduced incidence matrix of the tree ${\cal T}_k$ \\
${\cal E}^{{\cal T}_k}$ & set of edges in ${\cal T}_k$ \\
${\cal V}^{{\cal T}_k}$ & set of nodes in ${\cal T}_k$\\
path between $a, b\in {\cal V}^{{\cal T}_k}$ &
\begin{tabular}{|l}
subset of edges from ${\cal E}^{{\cal T}_k}$
s.t. each\\ node with edge in the subset, except\\ $a$ and $b$,
contributes exactly two edges\end{tabular} \\
${\cal E}_a^{{\cal T}_k}$ & path from $a$ to slack bus in ${{\cal T}_k}$ \\
$b$ is a descendant of $a$ & $a$ contributes ${\cal E}_b^{{\cal T}_k}$ \\
$b$ is the parent of $a$ & $(ab)\in {\cal E}_a^{{\cal T}_k}$ and $b$ contributes ${\cal E}_a^{{\cal T}_k}$ \\
$\theta_a, v_a$ & voltage phase and magnitude resp. at bus $a$ \\
$\theta, v$ & \begin{tabular}{|l}
Vector of non-substation voltage\\ phases and magnitudes resp. \end{tabular}\\
$\varepsilon_a$ & $=1-v_a$, voltage deviation at node $a$\\
$\varepsilon$ & $=1-v$, voltage deviation \\
$p_a, q_a$ & \begin{tabular}{|l} resp. active and reactive power\\ injection/consumption ($+/-$) at bus $a$\end{tabular} \\
$p, q$ & \begin{tabular}{|l}
Vector of non-substation \\ power injections/consumptions \end{tabular}\\
$\beta_{ab},g_{ab},r_{ab},x_{ab}$ & \begin{tabular}{|l} susceptance, conductance, resistance,\\ reactance resp. of edge $(ab)$ \end{tabular}\\
$\beta,g,r,x$ & \begin{tabular}{|l} diagonal matrix of line susceptances,\\ conductances, resistances, reactances resp.\end{tabular} \\
$\Sigma_y$ & matrix of second moments for variable $y$\\
$\Omega_y$ & matrix of covariances for variable $y$\\
$\mu_y$ & vector of means for variable $y$ \\
$H_y$ & \begin{tabular}{|l}
reduced weighted Laplacian matrix\\
with edge weights in $y$
\end{tabular}\\
$D^{{\cal T}_k}_a$ & Descendants (including itself) of $a$ in ${\cal T}_k$\\
$\cal M$ & set of unobserved nodes\\
\hline
\end{tabular}
\end{center}
\label{table1}
\end{table}

\subsection{Problem Formulation and Contribution}
\label{subsec:motivation}
We consider large distribution grids where the utility (observer) is unsure of the grid configuration because of insufficient or inaccurate switching data, perhaps caused by a recent system upset. Alternatively, we could take the point of view of a third party observer, who may be an aggregator or adversary, trying to extract the current forest configuration from available nodal measurements. We assume the current spanning forest configuration is kept intact sufficiently long for load profiles at grid nodes to attain a steady distribution (longer than the fluctuations but shorter than changes in the mean load).

We assume that the observer has access to nodal measurements, but not edge measurements -- an assumption consistent with the recent expansion of smart grid monitoring devices. Smart meters generally provide nodal voltages and power injections at fine spatial resolution, i.e. at the individual customer level, but they do not provide any edge flow data. Additional instrumentation is emerging for pole-mount or pad-mount transformers \cite{poles}, however, these new devices still only provide nodal voltages and aggregated customer power injections.\footnote{We use the term `power injection', `power consumption' and `load' interchangeably to denote the power profile at each interior (non-substation) node of the distribution system.} Some edge flow data is available to utilities, however, this is generally at a few select locations in the distribution grid, e.g. at the substation/root node, voltage regulators, reclosers, or other major utility equipment. These select locations may also have nodal and edge data from another emerging technology, i.e. distribution grid PMUs \cite{phadke1993synchronized}. However, we continue to restrict our input data to nodal values, which is consistent with the new, ubiquitous sensing provided by smart meters.

The nodal devices provide the observer with temporal samples of the nodal voltage magnitudes. The observer seeks to use these samples to learn the current configuration of switches that determine the `base-constrained spanning forest'. To supplement the voltage magnitude samples, the observer has historical information about statistics of the nodal consumption.

\section{Power Flow Models and Statistical Correlations}
\label{sec:pf+corr}
Our approach to the structure learning problem relies on linearized PF models on radial spanning forests that enable efficient reconstruction of the grid structure via a second-moment analysis. The most general of the two, termed the Linear Coupling (LC) model, ignores losses of active and reactive powers and consistently assumes small voltage magnitude and phase drops between connected nodes. For tree-like distribution grids, the LC-PF model becomes equivalent to the LinDistFlow PF model in \cite{89BWa,89BWb}. The second model considered in the paper, coined the DC-resistive model, corresponds to the special resistance dominating case of the LC-PF model. These PF models are described in more detail in Appendix \ref{app:PF}.

\subsection{Linear Coupled Power Flow (LC-PF) model}
\label{subsec:LC}
As noted in Eqs.~(\ref{PF_LPV_p},\ref{PF_LPV_q}) in Appendix \ref{app:PF}, the LC-PF model is derived from the general AC power flow model by assuming small voltage magnitude deviations and phase differences between neighboring buses in the grid. It is convenient to restate the linear equations in LC-PF model in matrix form as:
\begin{eqnarray}
&& p= H_g\varepsilon+H_{\beta}\theta, ~~ q= H_{\beta}\varepsilon-H_g\theta \label{LP-p}
\end{eqnarray}
where $p,q,\varepsilon$ and $\theta$ are defined in Table \ref{table1}. $H_g$ and $H_{\beta}$ are the weighted graph Laplacian matrices associated with forest ${\cal F}$ such that
\begin{eqnarray}
{\huge H}_g(a,b)&=\begin{cases}\sum_{c:(a,c) \in {\cal E}^{\cal F}}g_{ac} & \quad\text{if~} b = a\\
-g_{ab} & \quad\text{if~} (ab) \in {\cal E}^{\cal F}\\
 0 & \quad\text{otherwise}\end{cases} \label{Weighted Laplace}
\end{eqnarray}
$H_{\beta}$ has a similar structure with $g$-weights replaced by $\beta$-weights. The weighted graph Laplacians can be stated in terms of the directed incidence matrix $M$ as
\begin{eqnarray}
H_g = M^Tg^{\cal F}M, \quad H_{\beta} = M^T{\beta}^{\cal F}M
\end{eqnarray}
Here, $g^{\cal F}$ and ${\beta}^{\cal F}$ are diagonal matrices representing, respectively, line conductances and susceptances for edges within ${\cal F}$. $M$ is the edge to node directed incidence matrix of $\cal F$. See Fig. \ref{fig:picinc} for an example. Every row $m_{ab}$ in $M$ is equal to $\pm(e_a^T -e_b^T)$ and represents the directed edge $(ab)$, where the direction of an edge is chosen arbitrarily. $e_a \in \mathbb{R}^{N+k}$ is the standard basis vector associated with the vertex $a$, with $1$ at the $a^{th}$ position and zero everywhere else. We can combine Eqs.~(\ref{LP-p}) and express the complex power flows as:
\begin{eqnarray}
p+\hat{i}q = M^T(g^{\cal F}+ \hat{i}{\beta}^{\cal F})M(\varepsilon - \hat{i}\theta)\label{AC_pf3}
\end{eqnarray}

Both $H_\beta$ and $H_g$ are weighted graph Laplacians and are degenerate --- showing $K$ zero-eigenvalues associated with the freedom in fixing phase and voltage deviation (from nominal) at any node within each tree of the forest. It is natural to fix phases and voltages at the sub-stations making these `slack buses' $a_k$ for trees ${\cal T}_k$ of the (operational) forest, ${\cal F}$ such that $\theta_{a_k}=\varepsilon_{a_k}=0, $ for any $k$, $1 \leq k \leq K$. Formally, elimination of the set of $K$ sub-stations corresponds to elimination of $K$ components from all the vectors contributing Eqs.~(\ref{LP-p}), and reduction of $K$ rows and $K$ columns from the weighted Laplacian matrices. All the eigenvalues of the resulting reduced graph Laplacian matrices are thus strictly positive.

Without loss of generality, we will use the same notation for the original and reduced dimension variables $\theta, \varepsilon, p$ and $q$ and also refer to Eqs.~(\ref{LP-p},\ref{AC_pf3}) as applied to the reduced vectors of dimension $N\times 1$.
We will also keep notations, $H_{\beta}$ and $H_g$ for the reduced graph Laplacian matrices, and $M$ for the reduced incidence matrix respectively. The reduced $M$ has a block diagonal structure:
$M=\mbox{diag}(M_1,M_2,\cdots, M_K)$,
where, $M_k$ is the invertible reduced incidence matrix of tree ${\cal T}_k$ in $\cal F$. Thus, $M$ and correspondingly $H_\beta$ and $H_g$ are full rank, invertible and block-diagonal matrices. Inverting the linear non-degenerate Eqs.~(\ref{LP-p}) we arrive at
\begin{eqnarray}
\theta=& M^{-1}x^{\cal F}{M^{-1}}^Tp -M^{-1}r^{\cal F}{M^{-1}}^Tq = H^{-1}_{1/x}p - H^{-1}_{1/r}q\label{AC_pf4}\\
\varepsilon=& M^{-1}r^{\cal F}{M^{-1}}^Tp +M^{-1}x^{\cal F}{M^{-1}}^Tq = H^{-1}_{1/r}p + H^{-1}_{1/x}q \label{AC_pf5}
\end{eqnarray}
where $H_{1/r} \doteq M^T{r^{\cal F}}^{-1}M$ and $H_{1/x} \doteq M^T{x^{\cal F}}^{-1}M$. $r^{\cal F}$ and $x^{\cal F}$ are diagonal matrices representing, respectively, line resistances and reactances within the forest ${\cal F}$. Their relation to $g^{\cal F}$ and ${\beta}^{\cal F}$ are expressed in Eqs.~(\ref{g}).
\subsection{Relations between second moments}
\label{subsec:moments}
The real and reactive nodal power injections $p$ and $q$ in Eqs.~(\ref{AC_pf4},\ref{AC_pf5}) fluctuate because of exogenous processes, and their second moments are related by:
\begin{align}
 \mathbb{E}[\theta\theta^T] =& H^{-1}_{1/x}\mathbb{E}[pp^T]H^{-1}_{1/x} + H^{-1}_{1/r}\mathbb{E}[qq^T]H^{-1}_{1/r}\nonumber\\
&~- H^{-1}_{1/x}\mathbb{E}[pq^T]H^{-1}_{1/r}- H^{-1}_{1/r}\mathbb{E}[qp^T]H^{-1}_{1/x}\nonumber\\
\Rightarrow~ \Sigma_{\theta} =& H^{-1}_{1/x}\Sigma_pH^{-1}_{1/x} + H^{-1}_{1/r}\Sigma_qH^{-1}_{1/r} \nonumber\\
&~- H^{-1}_{1/x}\Sigma_{pq}H^{-1}_{1/r}- \left[H^{-1}_{1/x}\Sigma_{pq}H^{-1}_{1/r}\right]^T \label{thetacov}\\
\text{Similarly,}~~\Sigma_{\varepsilon} =& H^{-1}_{1/r}\Sigma_pH^{-1}_{1/r} + H^{-1}_{1/x}\Sigma_qH^{-1}_{1/x}\nonumber\\
 &~+ H^{-1}_{1/r}\Sigma_{pq}H^{-1}_{1/x}+\left[H^{-1}_{1/r}\Sigma_{pq}H^{-1}_{1/x}\right]^T\label{voltcov}\\
 \Sigma_{\theta\varepsilon} =& H^{-1}_{1/x}\Sigma_{p}H^{-1}_{1/r} - H^{-1}_{1/r}\Sigma_qH^{-1}_{1/x}\nonumber\\
 &~+ H^{-1}_{1/x}\Sigma_{pq}H^{-1}_{1/x} - H^{-1}_{1/r}\Sigma_{qp}H^{-1}_{1/r}\label{volangcov}
\end{align}
These formulas are the basis for reconstruction/learning analysis in the rest of the paper.

\subsection{DC-resistive model} 
\label{subsec:DC-res}
The DC-resistive PF model (see Appendix \ref{app:PF}) is an extremal case of the LC-PF model realized when line reactance can be ignored in comparison with resistance ($x/r\to 0$). The relation between the statistics of active powers and voltage second-order moments deviations reduces to
\begin{eqnarray}
\mathbb{E}[pp^T] = H_g\mathbb{E}[\varepsilon\varepsilon^T]H_g \Rightarrow~ \Sigma_{\varepsilon} =H^{-1}_g\Sigma_pH^{-1}_g \label{cov}
\end{eqnarray}
\label{sec:trends}
\begin{figure}[ht]
\squeezeup
\centering
\subfigure[]{\includegraphics[width=0.20\textwidth,height = .16\textwidth]{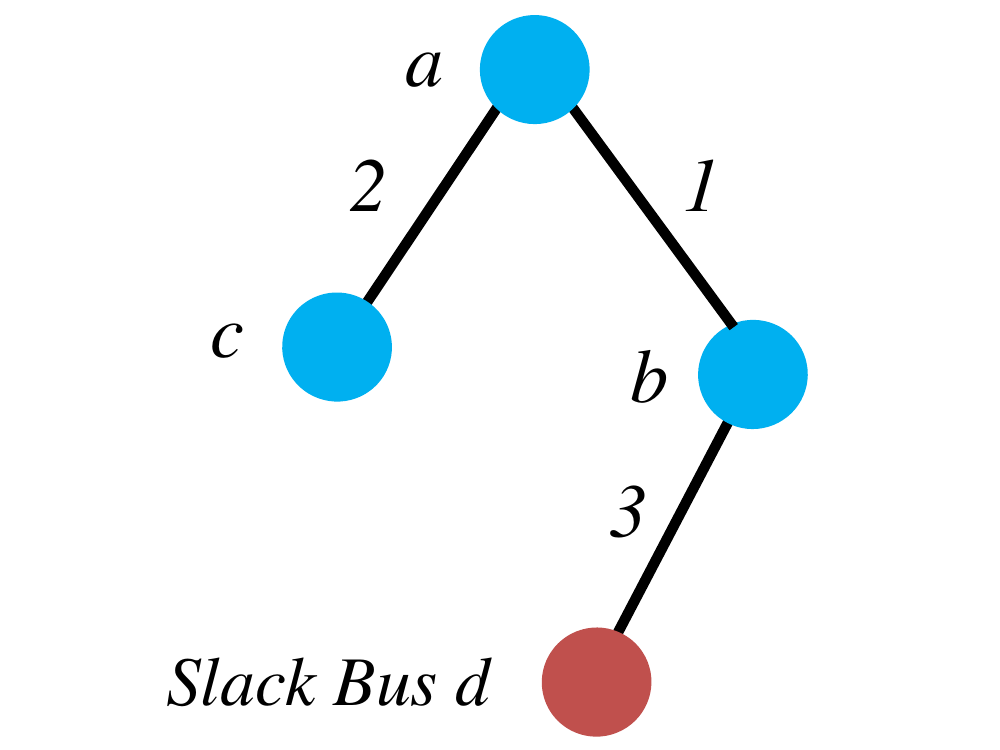}\label{fig:picinc3}}
\subfigure[]{\includegraphics[width=0.18\textwidth,height=0.10\textwidth]{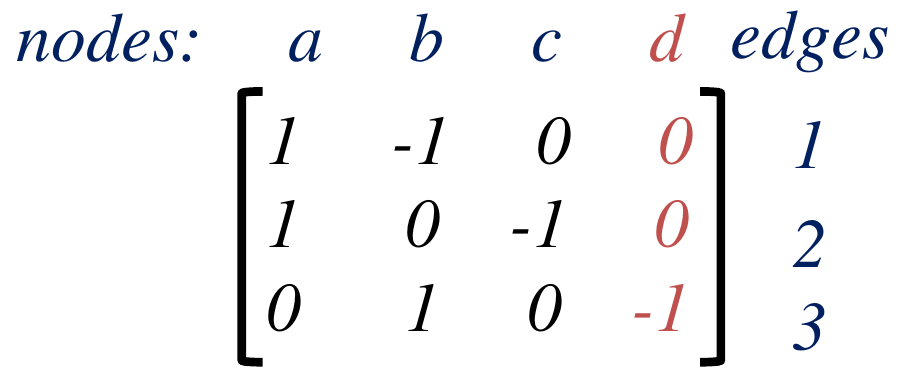}\label{fig:picinc_1}}
\squeezeup
\subfigure[]{\includegraphics[width=0.18\textwidth,height=0.10\textwidth]{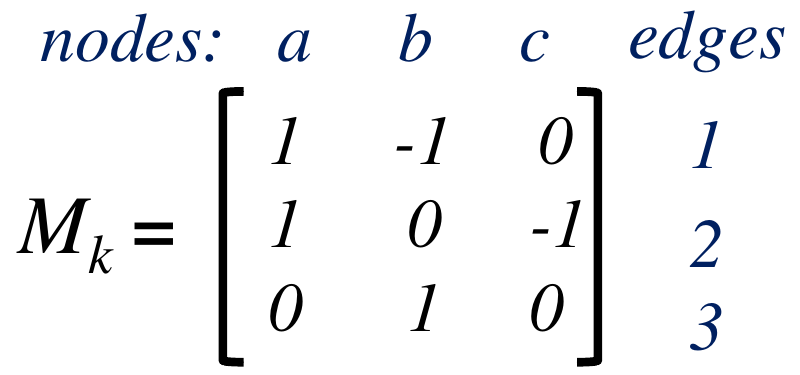}\label{fig:picinc1}}
\subfigure[]{\includegraphics[width=0.18\textwidth,height=0.10\textwidth]{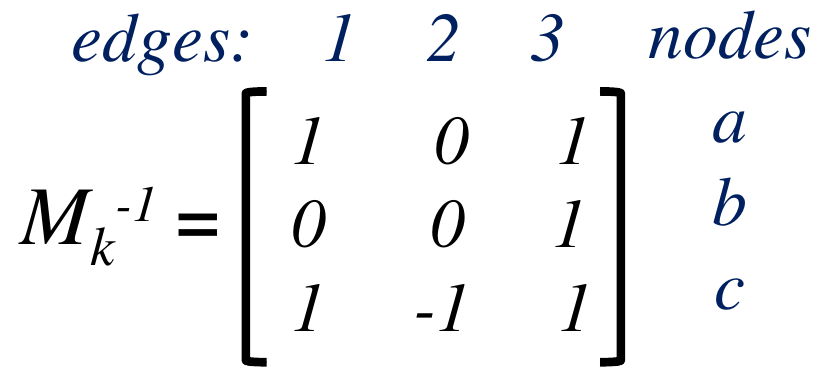}\label{fig:picinc2}}
\caption{Illustration of the reduced incidence matrix construction on the example of a tree with $4$ nodes. (a) Tree graph with four nodes ($a,b,c,d$) and three edges ($1,2,3$). (b) Complete directed incidence matrix. The ``directed" freedom in choosing edge orientations is fixed as follows, $1=(ab)$, $2=(ac)$, $3=(bd)$. (c) The column corresponding to the slack bus (bus $d$) is removed in the reduced incidence matrix $M_k$. (d) Inverse of the reduced incidence matrix $M_k$.
\label{fig:picinc}}
\vspace{-3mm}
\end{figure}

\section{Trends in Second Moments over Tree Networks}
We now derive key results related to the second moments in voltage magnitudes that arise from the properties of the forest $\cal F$. We denote the unique path from node $a$ to the slack bus in tree ${\cal T}_k$ by ${\cal E}_a^{{\cal T}_k}$. From \cite{68Resh}, the inverse of the reduced incidence matrix of a tree has the following special structure:

\begin{align} 
\squeezeup
{\huge M}_k^{-1}(a,r)=\begin{cases}1 & \text{if edge $r\in {\cal E}_a^{{\cal T}_k}$ is directed}\\
&\text{along path
from $a$ to slack bus},\\
-1 & \text{if edge $r\in {\cal E}_a^{{\cal T}_k}$ is directed}\\
&\text{against path from $a$ to slack bus}, \\
0 & \text{if edge~} r \not\in {\cal E}_a^{{\cal T}_k} \end{cases} \label{treeinv}
\squeezeup
\end{align}
Here, the direction of edge $r = (cd)$ is specified by its representative row $m_{cd}$ in the directed incidence matrix. For example, if $m_{cd} = e_c^T - e_d^T$, the edge is directed from node $c$ to node $d$, whereas for $m_{cd} = e_d^T - e_c^T$, the direction is from node $d$ to node $c$.

An immediate corollary of (\ref{treeinv}) is that $M^{-1}(a,r) = 0$ if edge $r$ and node $a$ lie on separate trees within the forest $\cal F$, a fact consistent with the block diagonal structure of $M$. Using (\ref{treeinv}) in $H_g^{-1} = M^{-1}{g^{\cal F}}^{-1}{M^{-1}}^T$, we derive for forest $\cal F$
\begin{align}
 H_g^{-1}(a,b)&= 0 \text{~if $a,b$ are on different trees ${\cal T}_{k}$ and}\label{Hinv0}\\
 H_g^{-1}(a,b)&= \sum_{r \in {\cal E}^{{\cal T}_k}} M^{-1}(a,r){g^{{\cal F}}}^{-1}(r,r)M^{-1}(b,r) \text{~if~} a,b \in {\cal T}_k\nonumber\\
&= \sum_{(cd) \in {\cal E}_a^{{\cal T}_k}\bigcap {\cal E}_b^{{\cal T}_k}} \frac{1}{g_{cd}} \text{~if~} a,b \in {\cal T}_k\label{Hinv}
\end{align}

Thus, $H_g^{-1}(a,b)$ is equal to \textit{the sum of the inverse conductances of lines that are common to the paths from both nodes to the slack bus.} If no such line exists, the corresponding entry in $H_g^{-1}$ is $0$. See Fig.~\ref{fig:picHinv1} for illustration. Similar results hold for other measurement matrices like
\begin{eqnarray}
 H^{-1}_{1/r}(a,b) &&=\begin{cases} \sum_{(cd) \in {\cal E}_a^{{\cal T}_k}\bigcap {\cal E}_b^{{\cal T}_k}} r_{cd} \text{~~if nodes~} a,b \in {\cal T}_k\\
 0 ~~ \text{otherwise,} \end{cases}\label{Hrxinv}
\end{eqnarray}

Let $D^{{\cal T}_k}_a$ denote the set of descendants of node $a$ within the tree ${\cal T}_k$. We call $b$ a descendent of $a$, if $a$ lies on the (unique) path from $b$ to the slack bus of ${\cal T}_k$, also including $a$ itself in the set of its descendants. We call $b$ the parent of $a$ within ${\cal T}_k$ (there can only be one) if $(ab)\in{\cal T}_k$ and $a$ is an immediate descendant of $b$ as illustrated in Fig. \ref{fig:descendant}.
\begin{figure}[!bt]
\centering
\subfigure[]{\includegraphics[width=0.20\textwidth]{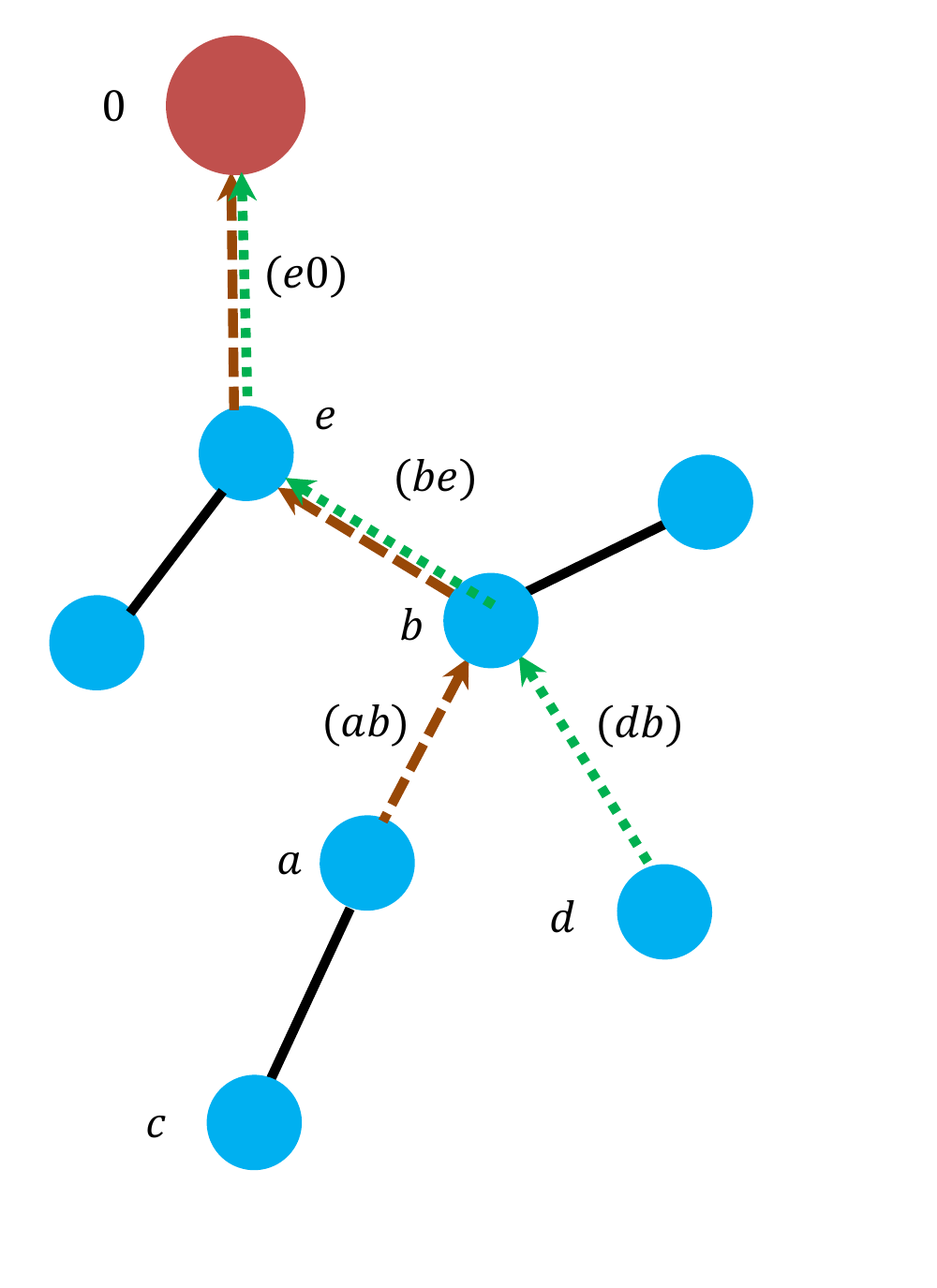}\label{fig:picHinv1}}\hspace{.6cm}
\subfigure[]{\includegraphics[width=0.20\textwidth]{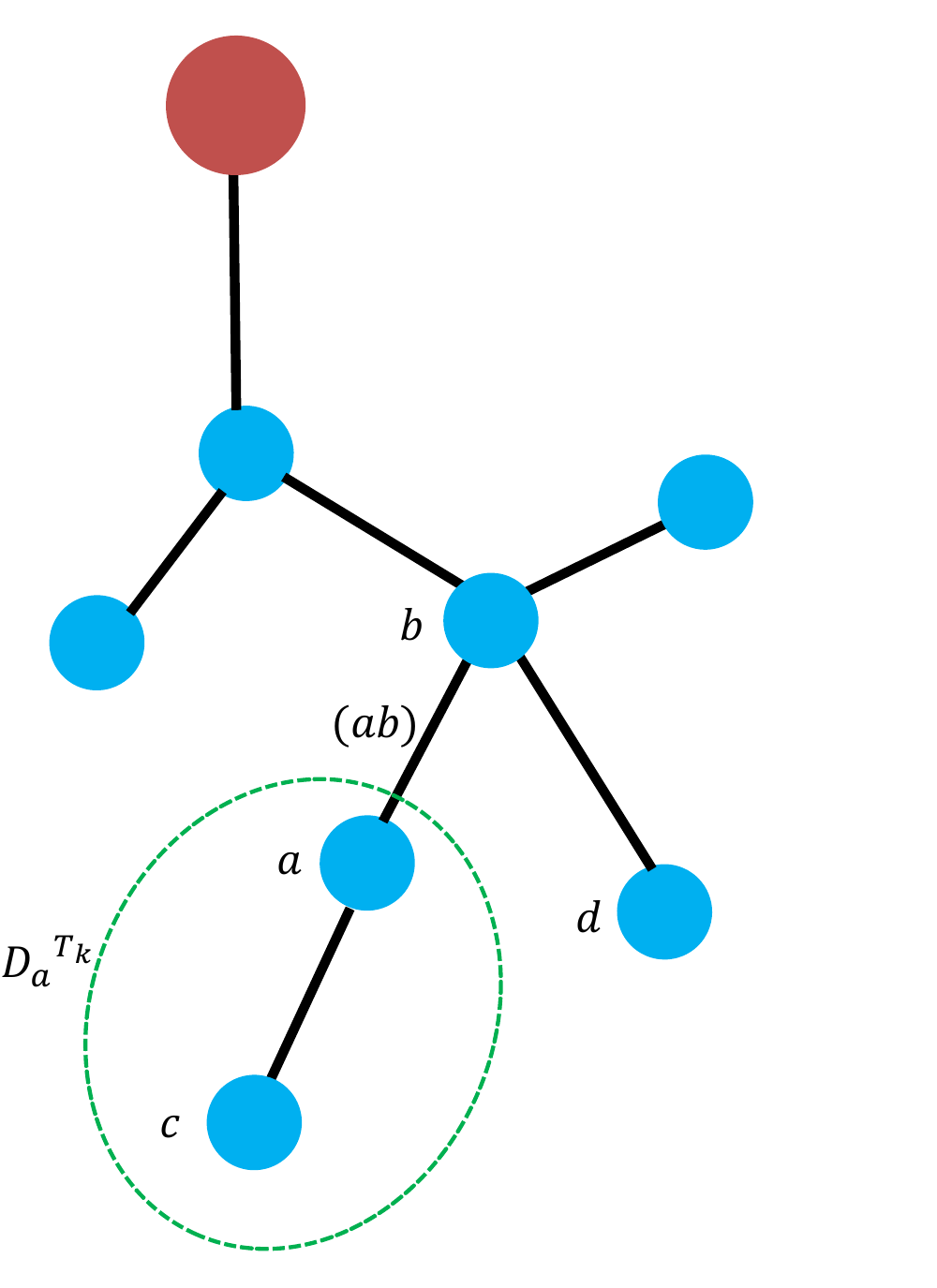}\label{fig:descendant}}
\caption{Schematic layout of a distribution grid tree ${\cal T}_k$. The sub-station node represented by large red node is the slack bus. (a) Dotted lines represent the paths from nodes $a$ and $d$ to the slack bus. The common edges on those paths are edges $(be)$ and $(e0)$. Thus, $H_g^{-1}(a,d) = 1/g_{be}+ 1/g_{e0}$. (b) the set of descendants of $a$, denoted by $D_a^{{\cal T}_k}$. Here, nodes $a$ and $c$ are descendants of both nodes $a$ and $b$, while node $d$ is not a descendant of $a$ but only of node $b$.
\label{fig:picHinv}}
\vspace{-3mm}
\end{figure}

The following statement holds.
\begin{lemma}\label{Lemmadiff}
For two nodes, $a$ and its parent $b$, in tree ${\cal T}_k$
\begin{align} 
{\huge H}_g^{-1}(a,c)-{\huge H}_g^{-1}(b,c)&&=\begin{cases}\frac{1}{g_{ab}} & \quad\text{if node $c \in D^{{\cal T}_k}_a$}\\
0 & \quad\text{otherwise,} \end{cases} \label{Hdiff}
\end{align}
\end{lemma}
\begin{proof}
For any node $c$ which belongs to a tree not containing nodes $a$ and $b$, $H_g^{-1}(a,c)-H_g^{-1}(b,c) = 0$ according to (\ref{Hinv0}). Now, focus on nodes contained, together with $a$ and $b$, within ${\cal T}_k$. Since $b$ is $a$'s parent, ${\cal E}_a^{{\cal T}_k} = {\cal E}_b^{{\cal T}_k}\bigcup \{(ab)\}$ and we derive (also validating on the illustrative example in Fig.~\ref{fig:descendant}) that for any node $c$ in tree ${\cal T}_k$,
\begin{align*}
{\cal E}_a^{{\cal T}_k}\bigcap {\cal E}_c^{{\cal T}_k} &= {\cal E}_b^{{\cal T}_k}\bigcap {\cal E}_c^{{\cal T}_k} \quad &&\text{if node $c \not\in D^{{\cal T}_k}_a$}\\
{\cal E}_a^{{\cal T}_k}\bigcap {\cal E}_c^{{\cal T}_k} &= [{\cal E}_b^{{\cal T}_k}\bigcap {\cal E}_c^{{\cal T}_k}] \bigcup \{(ab)\} \quad &&\text{if node $c \in D^{{\cal T}_k}_a$}
\end{align*}
resulting in Eq.~(\ref{Hdiff}).
\end{proof}

We now prove our main results regarding trends in second moments of deviations in voltage magnitudes ($\varepsilon$) along any tree in the network. The results are conditioned on the following assumption regarding correlations in power injections at the non-substation buses.

\textbf{Assumption $1$}: For any two buses $a$ and $b$ drawing power from the same distribution sub-station, $\Sigma_p(a,b)>0, \Sigma_q(a,b)>0, \Sigma_{pq}(a,b)>0$.

Note that this assumption holds, in particular, if the overall node balance is such that each non-substation node $a$ always consumes strictly more than it produces in active and reactive powers, i.e. $p_a <0, q_a <0$. This assumption is certainly true in any distribution grid with small and/or moderate penetration of renewables \cite{pedersen2008load}. However, the assumption is also reasonable for a system with significant penetration of generation which is still dominated in average by the consumption.

\begin{theorem} \label{Theorem1}
If node $a \neq b$ is a descendant of node $b$ within tree ${\cal T}_k$, then for the DC-resistive model, $\Sigma_{\varepsilon}(a,a) > \Sigma_{\varepsilon}(b,b)$.
\end{theorem}
\begin{proof}
We first show that for any node $a$ and its parent $b$, $\Sigma_{\varepsilon}(a,a) > \Sigma_{\varepsilon}(b,b)$. Consider $\Sigma_{\varepsilon}(a,a) - \Sigma_{\varepsilon}(a,b)$. From Eq.~(\ref{cov}), we derive
\begin{align}
\Sigma_{\varepsilon}(a,a) - \Sigma_{\varepsilon}(a,b) =& \sum_{c,d}H_g^{-1}(a,c)\Sigma_p(c,d)\nonumber\\
&~~\left(H_g^{-1}(a,d)-H_g^{-1}(b,d)\right) \label{uselater}\\
(\text{using Lemma \ref{Lemmadiff}})=& \sum_{c,d}H_g^{-1}(a,c)\Sigma_p(c,d)\frac{1}{g_{ab}}\textbf{1}(d \in D^{{\cal T}_k}_a) \nonumber\\
\Rightarrow \Sigma_{\varepsilon}(a,a) - \Sigma_{\varepsilon}(a,b) >& 0 ~(\text{using Assumption $1$})\label{equ1}\\
\text{Also,}\Sigma_{\varepsilon}(a,b) - \Sigma_{\varepsilon}(b,b) &= \smashoperator[lr]{\sum_{c \in D^{{\cal T}_k}_a,d}}H_g^{-1}(b,d)\Sigma_p(c,d)\frac{1}{g_{ab}}  > 0\label{equ2}
\end{align}

Combining Eqs.~(\ref{equ1}) and (\ref{equ2}) we derive $\Sigma_{\varepsilon}(a,a) > \Sigma_{\varepsilon}(a,b) > \Sigma_{\varepsilon}(b,b)$. Since node $a$ is a descendant of node $b$, there is a path $a, c_1,...c_r,b$, such that each node in the path is a parent of its predecessor. Then, we derive $\Sigma_{\varepsilon}(a,a) > \Sigma_{\varepsilon}(c_1,c_1) >...> \Sigma_{\varepsilon}(c_r,c_r) > \Sigma_{\varepsilon}(b,b)$.
\end{proof}

The following theorem is the LC-PF version of Theorem \ref{Theorem1}.
\begin{theorem} \label{Theorem1_LC}
If node $a \neq b$ is a descendant of node $b$ on tree ${\cal T}_k$ in forest $\cal F$, then $\Sigma_{\varepsilon}(a,a) > \Sigma_{\varepsilon}(b,b)$ in the LC-PF model.
\end{theorem}
\begin{proof}
Consider Eq.~(\ref{voltcov}). Notice that the right hand side has four constituent terms ($H^{-1}_{1/r}\Sigma_pH^{-1}_{1/r}$, $H^{-1}_{1/x}\Sigma_qH^{-1}_{1/x}$, $H^{-1}_{1/r}\Sigma_{pq}H^{-1}_{1/x}$ and $H^{-1}_{1/x}\Sigma_{qp}H^{-1}_{1/r}$). We denote each of these terms by $\Sigma^j_{\varepsilon}$ where $j
\in \{1,2,3,4\}$. For each individual term, applying Assumption $1$ and the analysis in Theorem \ref{Theorem1}, we find that $\Sigma^j_{\varepsilon}(a,a) > \Sigma^j_{\varepsilon}(b,b)$ if node $a$ is a descendant of node $b$, other than itself. Thus, the statement also holds for the sum.
\end{proof}

We now focus on evaluating the term $\mathbb{E}[(\varepsilon_a-\varepsilon_b)^2]$, which is the expected value of the squared difference between two node voltage deviations ($\varepsilon$). For any two nodes $a$ and $b$, the DC-resistive model yields:
\begin{align}
\mathbb{E}[(\varepsilon_a-\varepsilon_b)^2] &= \Sigma_{\varepsilon}(a,a) - \Sigma_{\varepsilon}(a,b) - \Sigma_{\varepsilon}(b,a) + \Sigma_{\varepsilon}(b,b)\nonumber\\
(\text{using (\ref{uselater})}) &= \sum_{c,d}H_g^{-1}(a,c)\Sigma_p(c,d)\left(H_g^{-1}(a,d)- H_g^{-1}(b,d)\right)\nonumber\\
 &~- \sum_{c,d}H_g^{-1}(b,c)\Sigma_p(c,d)\left(H_g^{-1}(a,d)- H_g^{-1}(b,d)\right)\nonumber\\
&= \sum_{c,d}\left(H_g^{-1}(a,c)- H_g^{-1}(b,c)\right)\Sigma_p(c,d)\nonumber\\
&\quad\quad~\left(H_g^{-1}(a,d)- H_g^{-1}(b,d)\right)\label{diffsq}
\end{align}

Our next Lemma follows directly by applying Lemma \ref{Lemmadiff} to Eq.~(\ref{diffsq}).
\begin{lemma}\label{Lemmadiffsq}
For two nodes, $a$ and its parent $b$ belonging to tree ${\cal T}_k$, in the DC-resistive model, we derive $\mathbb{E}[(\varepsilon_a-\varepsilon_b)^2] = \sum_{c,d \in D^{{\cal T}_k}_a} \frac{1}{g_{ab}^2}\Sigma_p(c,d)$
\end{lemma}

For the LC-PF model, we evaluate the expression $\mathbb{E}[(\varepsilon_a-\varepsilon_b)^2]$ as well. In this case, for two nodes, $a$ and its parent $b$, that lie on tree ${\cal T}_k$, we arrive at
\squeezeup
\begin{eqnarray}
\mathbb{E}[(\varepsilon_a-\varepsilon_b)^2] &&= \Sigma_{\varepsilon}(a,a) - \Sigma_{\varepsilon}(a,b) - \Sigma_{\varepsilon}(b,a) + \Sigma_{\varepsilon}(b,b)\nonumber
\end{eqnarray}
where $\Sigma_{\varepsilon}$ is given by Eq.~(\ref{voltcov}). Let $\Sigma^1_{\varepsilon}$ and $\Sigma^2_{\varepsilon}$ represent the symmetric terms $H^{-1}_{1/r}\Sigma_pH^{-1}_{1/r}$ and $H^{-1}_{1/x}\Sigma_qH^{-1}_{1/x}$ respectively in $\Sigma_{\varepsilon}$. Extending the result of Lemma \ref{Lemmadiffsq},
we derive
\begin{align}
\Sigma^1_{\varepsilon}(a,a) -\Sigma^1_{\varepsilon}(a,b) - \Sigma^1_{\varepsilon}(b,a) + \Sigma^1_{\varepsilon}(b,b) = \smashoperator[lr]{\sum_{c,d \in D^{{\cal T}_k}_a}} r_{ab}^2\Sigma_p(c,d)\label{firstterm}\\
\Sigma^2_{\varepsilon}(a,a) -\Sigma^2_{\varepsilon}(a,b) - \Sigma^2_{\varepsilon}(b,a) + \Sigma^2_{\varepsilon}(b,b) = \smashoperator[r]{\sum_{c,d \in D^{{\cal T}_k}_a}} x_{ab}^2\Sigma_q(c,d)\label{secondterm}
\end{align}

Similarly, let $\Sigma^3_{\varepsilon}= H^{-1}_{1/r}\Sigma_{pq}H^{-1}_{1/x}$ and $\Sigma^4_{\varepsilon}= {\Sigma^3}^T_{\varepsilon}= \left[H^{-1}_{1/r}\Sigma_{pq}H^{-1}_{1/x}\right]^T$, the non-symmetric terms in $\Sigma_{\varepsilon}$. Using Lemma \ref{Lemmadiff} with Eq.~(\ref{diffsq}), we get
\begin{align}
&\Sigma^3_{\varepsilon}(a,a) -\Sigma^3_{\varepsilon}(a,b) - \Sigma^3_{\varepsilon}(b,a) + \Sigma^3_{\varepsilon}(b,b)
=\smashoperator[r]{\sum_{c,d\in D^{{\cal T}_k}_a}}r_{ab}\Sigma_{pq}(c,d)x_{ab}\nonumber\\
&=\Sigma^4_{\varepsilon}(a,a) -\Sigma^4_{\varepsilon}(a,b) - \Sigma^4_{\varepsilon}(b,a) + \Sigma^4_{\varepsilon}(b,b)
\label{thirdterm}
\end{align}
Combining Eqs.~(\ref{firstterm},\ref{secondterm},\ref{thirdterm}) we arrive at the following Lemma.
\begin{lemma}\label{LemmadiffsqLC}
In the LC-PF model, $\mathbb{E}[(\varepsilon_a-\varepsilon_b)^2] = \sum_{c,d \in D^{{\cal T}_k}_a} r_{ab}^2\Sigma_p(c,d)+x_{ab}^2\Sigma_q(c,d)+2r_{ab}x_{ab}\Sigma_{pq}(c,d)$, holds
for a node $a$ and its parent $b$ belonging to the (operational) tree ${\cal T}_k$.
\end{lemma}

\section{Learning Structure of Base-Constrained Spanning Forest}
\label{sec:learn}

Here, we propose Algorithm $1$ to learn the structure of the distribution network using properties of voltage deviations for the LC-PF model. The polynomial time algorithm, based on the Theorems proved in the previous section, requires positivity of the correlation between nodal power injections (Assumption $1$). The Algorithm is also agnostic to the probability distribution of active and reactive power injections. To reconstruct the Base-Constrained Spanning Forest (${\cal F}=\cup_{k=1,\cdots,K}{\cal T}_k$), Algorithm $1$ takes as input `$m$' measurements of nodal voltage deviations. These measurements are used to create the empirical voltage deviation second moment matrix $\Sigma_{\varepsilon}$. The voltage deviations $\varepsilon^j_a$ at the substation nodes are assumed to be zero which implies the elements of the row $\Sigma_{\varepsilon}(a,:)$=0 for each of the $K$ substation nodes. The observer has prior information (or estimates using power injection measurements) for the true second moment matrix of power injections $\Sigma_p$ for non-substation nodes.

\begin{algorithm}
\caption{Base Constrained Spanning Forest Learning: LC-PF Model}
\textbf{Input:} True $\Sigma_p, \Sigma_q$ and $\Sigma_{pq}$, $m$ voltage deviation observations $\varepsilon^j, 1\leq j \leq m$, all line resistances $r$ and reactances $x$\\
\begin{algorithmic}[1]
\State Compute $\Sigma_{\varepsilon}(a,a) = \sum_{j = 1}^m\varepsilon^j_a\varepsilon^j_a/m$ for all nodes $a$.
\State Undiscovered Set $U \gets \{1,2,...,N+K\}$, Leaf Set $L \gets \phi$, Descendant Sets $D_a \gets \{a\} \forall$ nodes $a$.
\While {($U \neq \phi)$}
\State $b^* \gets \max_{b \in U} \Sigma_{\varepsilon}(b,b)$ \label{step2_1}
 \ForAll{$a \in L$}
\If {$\sum_{j=1}^m(\varepsilon^j_a-\varepsilon^j_{b^*})^2/m = \sum_{c,d \in D_a} r_{ab}^2\Sigma_p(c,d)+x_{ab}^2\Sigma_q(c,d)+2r_{ab}x_{ab}\Sigma_{pq}(c,d)$} \label{step2_2}
 \State Draw edge between nodes $a$ and $b^*$
 \State $D_{b^*} \gets D_{b^*} \bigcup D_a$
 \State $L \gets L - \{a\}$ \label{step2_3}
\EndIf
\EndFor
 \State $L \gets L \bigcup \{b^*\}$ \label{step2_4}
\EndWhile
\end{algorithmic}
\end{algorithm}

\textbf{Algorithm 1 Overview:} We reconstruct each tree within the distribution grid forest sequentially moving from the leaves to the root nodes. At every stage, $U$ represents the set of undiscovered nodes that are not part of the current reconstructed tree while $L$ represents the set of `current leaves' (nodes that are in the current reconstructed tree but with undiscovered parents). At each iteration, Step \ref{step2_1} selects the node $b^*$ from set $U$ with the largest second moment of voltage deviation. Next Step \ref{step2_2} adds edges between node $b^*$ and nodes in set $L$ of the growing tree using Lemma \ref{LemmadiffsqLC}. In the ideal case when infinitely many voltage magnitude samples are collected, second-order moments of the power injections satisfy the relation in Lemma \ref{LemmadiffsqLC}. However, we have a finite number of samples. Thus the presence of an edge is determined in Algorithm $1$ by checking if the relative difference between the reals on the left and right sides of the condition in Step \ref{step2_2} is less than a predefined tolerance, $\tau$:
\begin{align}
1 - \Biggl|\frac{\sum_{j=1}^m(\varepsilon^j_a-\varepsilon^j_{b^*})^2/m}{\sum_{c,d \in D_a} r_{ab}^2\Sigma_p(c,d)+x_{ab}^2\Sigma_q(c,d)+2r_{ab}x_{ab}\Sigma_{pq}(c,d)}\Biggr| < \tau \label{testcond}
\end{align}
Steps \ref{step2_3} and \ref{step2_4} update the set of current leaves $L$ before repeating the reconstruction steps with a new undiscovered node.\\

\textbf{Algorithm Complexity:} Ignoring complexity of computing the second moments in Steps \ref{step2_1} and \ref{step2_2} (part of the data pre-processing), there are $N+K$ steps ($N$ load nodes and $K$ substation nodes) in the `while' loop, and at most $N$ comparisons in the `for' loop for each node. Therefore, the worst-case complexity of this algorithm is $O(N^2+ NK)$.

Note that we use LC-PF model in Algorithm $1$. To design the DC-resistive version of Algorithm $1$, the condition in Step \ref{step2_2} should be replaced with the result in Lemma \ref{Lemmadiffsq}. (Required modifications are straightforward and thus their description is omitted.)

\section{Experiments}
\label{sec:experiments}
We perform a set of numerical experiments to test and demonstrate the performance of Algorithm $1$ in extracting the operational radial forest $\cal F$ from meshed ``as-designed'' distribution networks $\cal G$. We remind the reader that the observer in Algorithm $1$ has information of the full graph $\cal G$, the impedance (resistance and reactance) of all lines (operational or open) as well as the number of connected substation buses. Further, true second moments of active and reactive power injections at each non-substation node are assumed to be known. The set of measurements available as input with the observer comprises of deviations in voltage magnitudes ($\varepsilon$) at the grid nodes.

Table \ref{table_testcases} summarizes the distribution grid test systems by the number of load busses, number of substation busses, and the number of tie switches. Additional information on these test systems \cite{testcase1,testcase2,testcase3} can be found online at \cite{radialsource}. In normal operation, each test grid consists of nodes in a forest- or tree-like configuration $\cal{F}$ with open tie-switches. We construct the complete meshed network $\cal{G}$ by closing all the tie-switches. To test the scalability of our algorithms, we increase the number of possible forest configuration by introducing several additional non-operational lines into each system as noted in Table \ref{table_testcases}. Although these contribute to $\cal{G}$, they are kept open and do not contribute to power flows in the operation forest $\cal{F}$. The impedances (reactances and resistances) of the additional lines are generated by assigning random values uniformly between the minimum and maximum impedances of the operational lines. Fig.~\ref{fig:cases} displays the test networks $\cal{G}$ (solid and dashed lines) and the respective operational forests $\cal{F}$ (solid lines only).

\begin{table}[ht]
\caption{Summary of the tested distribution grids}
\begin{center}
\begin{tabular}{|p{1cm}|p{3cm}|p{2cm}|p{1cm}|}
\hline
Test Case & Number of buses / substations / tie-switches & Non-operational lines added & Source\\
\hline
$bus\_13\_3$ & $13/3/3$ & $10$ & \cite{testcase1} \\\hline
$bus\_29\_1$ & $29/1/1$ & $20$ & \cite{testcase2} \\\hline
$bus\_83\_11$ & $83/11/13$ & $30$ & \cite{testcase3} \\\hline
\end{tabular}
\end{center}
\label{table_testcases}
\end{table}

\begin{figure}[!bt]
\centering
\subfigure[]{\includegraphics[width=0.33\textwidth,height = .25\textwidth]{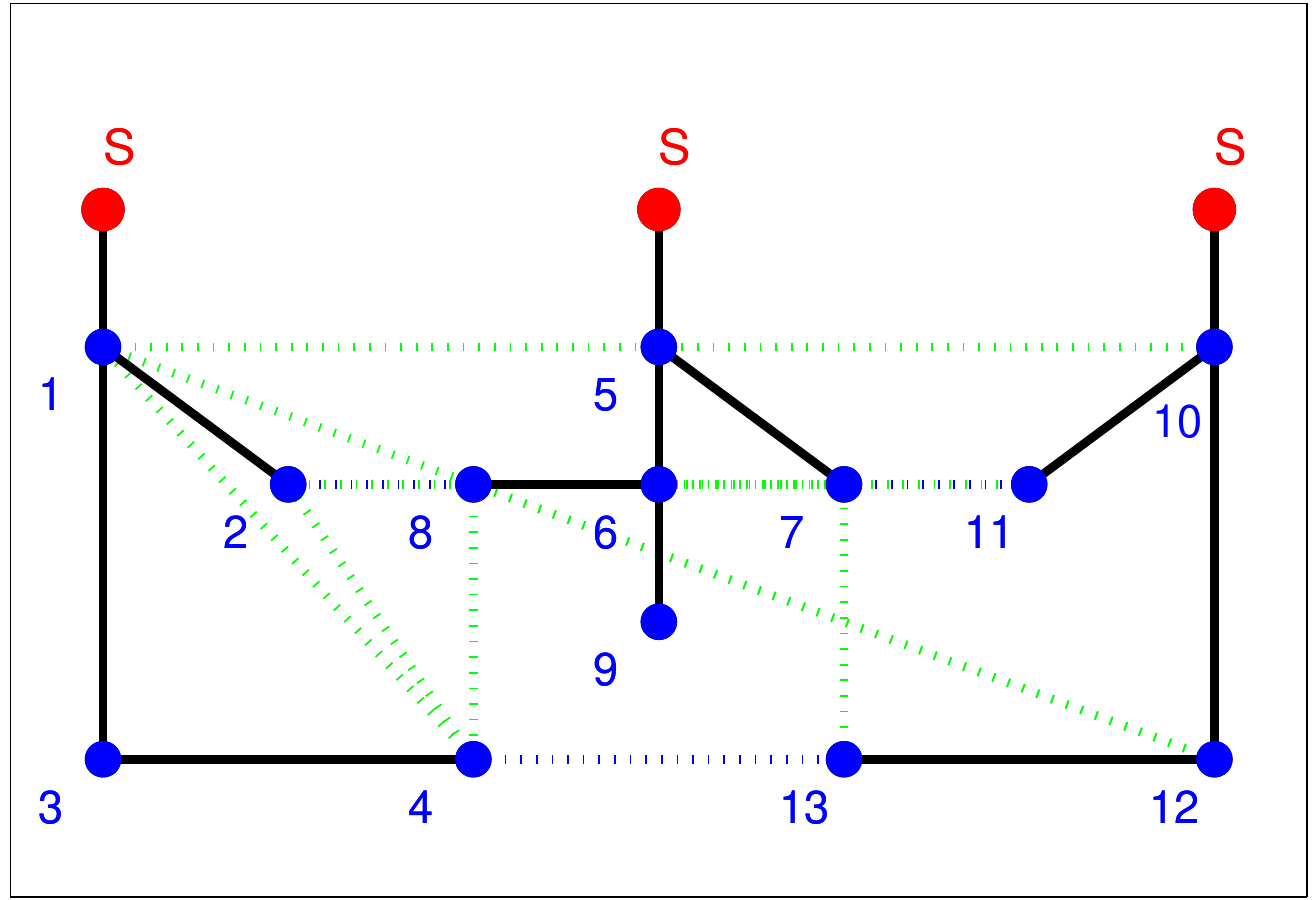}\label{fig:algo6case1}}
\subfigure[]{\includegraphics[width=0.33\textwidth,height = .25\textwidth]{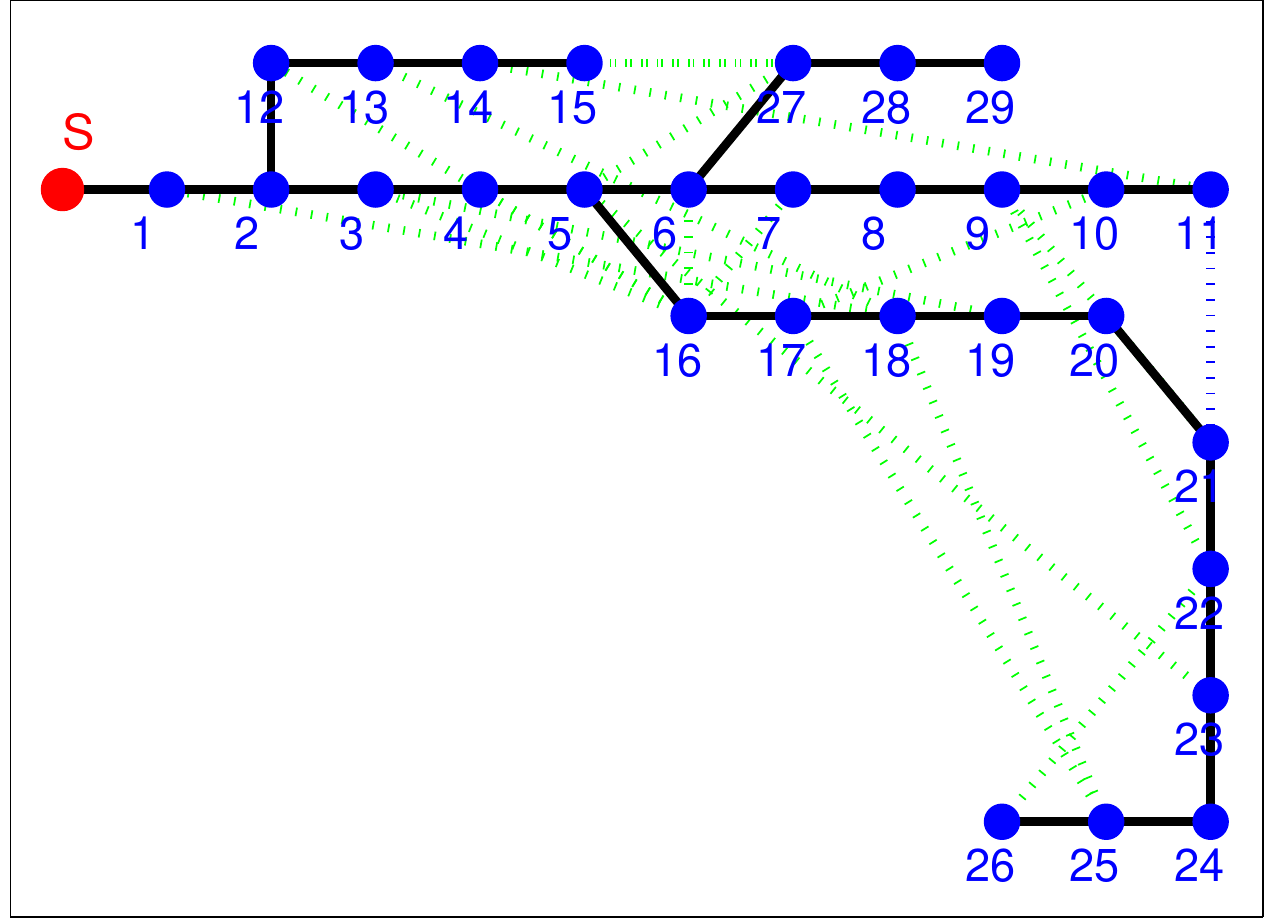}\label{fig:algo6case2}}
\subfigure[]{\includegraphics[width=0.33\textwidth,height = .25\textwidth]{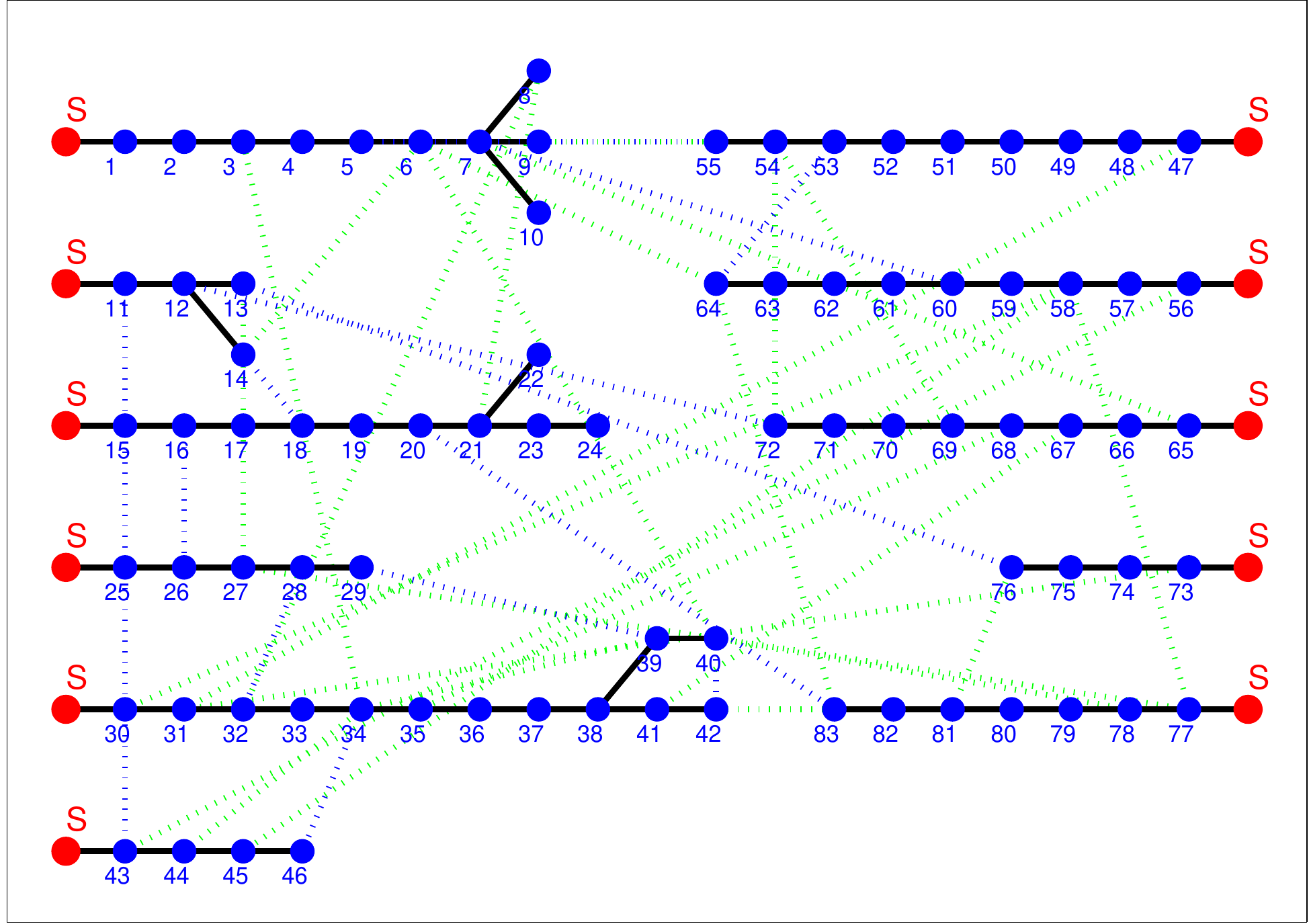}\label{fig:algo6case3}}
\vspace{-.25cm}
\caption{Layouts of distribution grids, also showing operational cases tested, in accordance with description in the text and summary in Table \ref{table_testcases}. The red circles represent substations (marked as $S$). The blue circles represent load nodes that are numbered. Black lines represent operational edges, while dotted blue lines represent open tie-switches. The additional lines are represented by dotted green lines. (a) $bus\_13\_3$ test case (b) $bus\_29\_1$ test case (c) $bus\_83\_11$ test case
\label{fig:cases}}
\vspace{-3mm}
\end{figure}

For each numerical experiment on a grid from Table \ref{table_testcases}, we pick an operational spanning forest $\cal{F}$ by opening tie switches. We also choose the statistics of the injections at each node using Gaussian distributions, unless otherwise specified. These distributions are used to generate multiple power injection samples, and from each vector-valued sample, we solve a PF to compute voltages and phases at every node in the network with the voltages at the substations fixed (i.e. they are slack busses). Averaging over all PF solutions, we compute empirical correlations of voltage magnitudes and phases. Using only these correlations, we run our algorithms and compare the resulting reconstruction with the actual operational configuration. In the reconstruction, we assume the observer has access to the resistance and reactance of all the lines in $\cal{G}$. All powers and voltages are presented in per unit (p.u.) values.

Figs. \ref{fig:algo2case1}-\ref{fig:algo2case3} display the accuracy of Algorithm $1$ for the different test grids from Table \ref{table_testcases}. The relative error is defined to be the number of mislabeled lines (connected when actually open and vice versa) divided by the size of the operational edge set. The relative error is averaged by computing many reconstructions using the same nodal power injection distributions. Different curves (colors) in Fig.~\ref{fig:algo2case1} show the effect of changing the tolerance $\tau$ in Eq.~\ref{testcond}.

The average fractional error in Figs. \ref{fig:algo2case1}-\ref{fig:algo2case3} decays exponentially with the number of samples used by the observer in the learning algorithm. For the tolerance values considered here, the majority of structural errors arise due to connected nodes not satisfying condition (\ref{testcond}) and hence being labelled as open. The decay is then intuitive as an increase in the sample size makes empirical moments in voltage magnitudes approximate their true values better which in turn leads to an increase in the number of operational lines satisfying (\ref{testcond}) and being correctly identified. On the other hand, if a sufficiently large value of $\tau$  is used, condition (\ref{testcond}) will be relaxed and possibly be satisfied even by unconnected nodes. In such a case, a majority of errors will be recorded due to open lines being incorrectly labelled as operational. As errors of this type (open edges classified as operational) does not improve with the number of measurement samples, the average fractional errors will not decay with the sample size. This is elucidated in Fig. \ref{fig:algo2_less}, where the $bus\_13\_3$ structure is learnt, in the presence of $50$ non-operational lines. Note that for larger values of $\tau$ in Fig. \ref{fig:algo2_less}, the errors do not decay with the sample size whereas for smaller values, they do as justified in the preceding discussion. For larger sample sizes, a smaller value of $\tau$ is indeed preferable as observed in Fig. \ref{fig:algo2_less}.

\begin{figure}[!bt]
\centering
\subfigure[]{\includegraphics[width=0.42\textwidth,height = .35\textwidth]{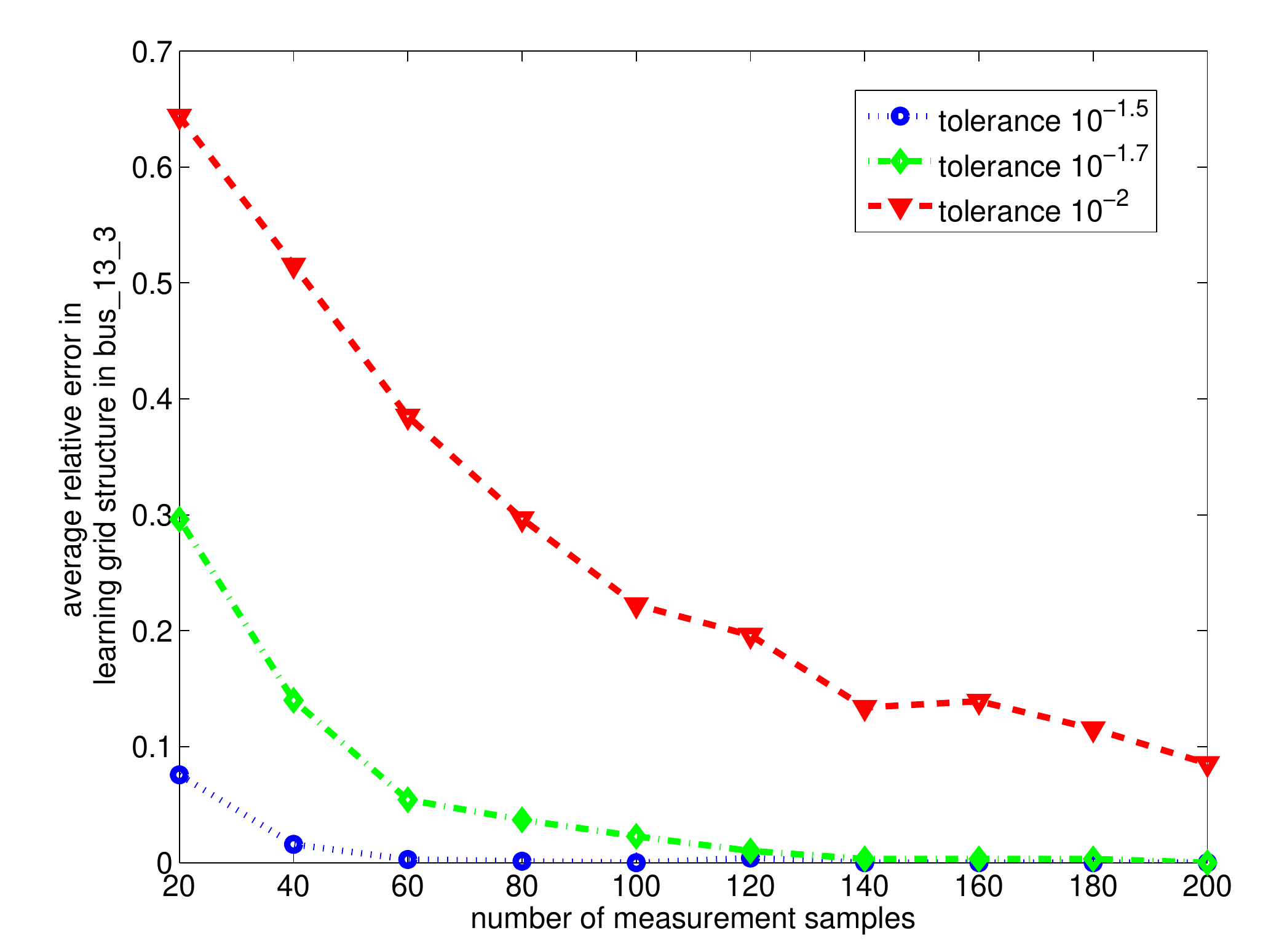}\label{fig:algo2case1}}
\subfigure[]{\includegraphics[width=0.42\textwidth,height=0.35\textwidth]{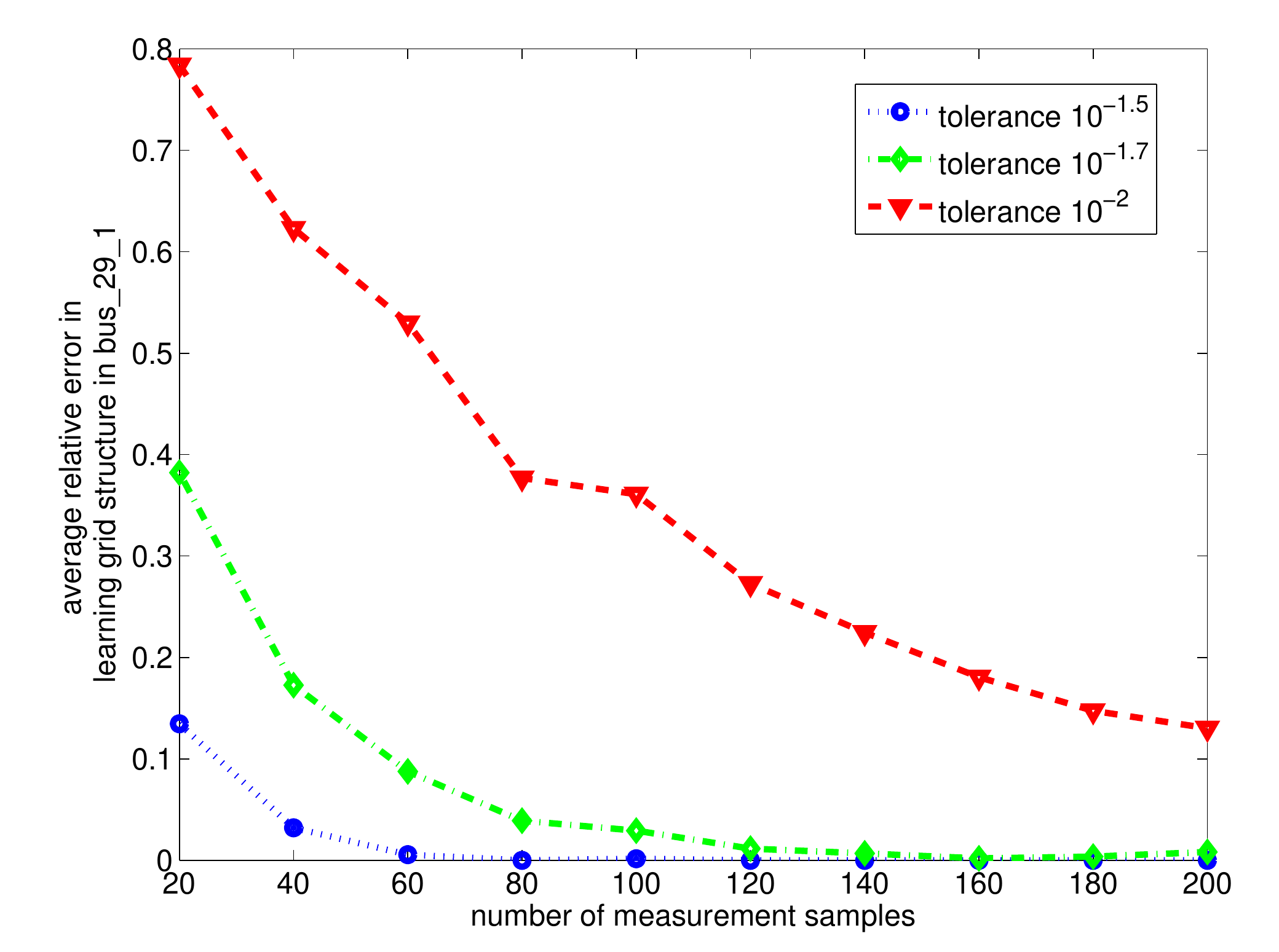}\label{fig:algo2case2}}
\subfigure[]{\includegraphics[width=0.42\textwidth,height = .35\textwidth]{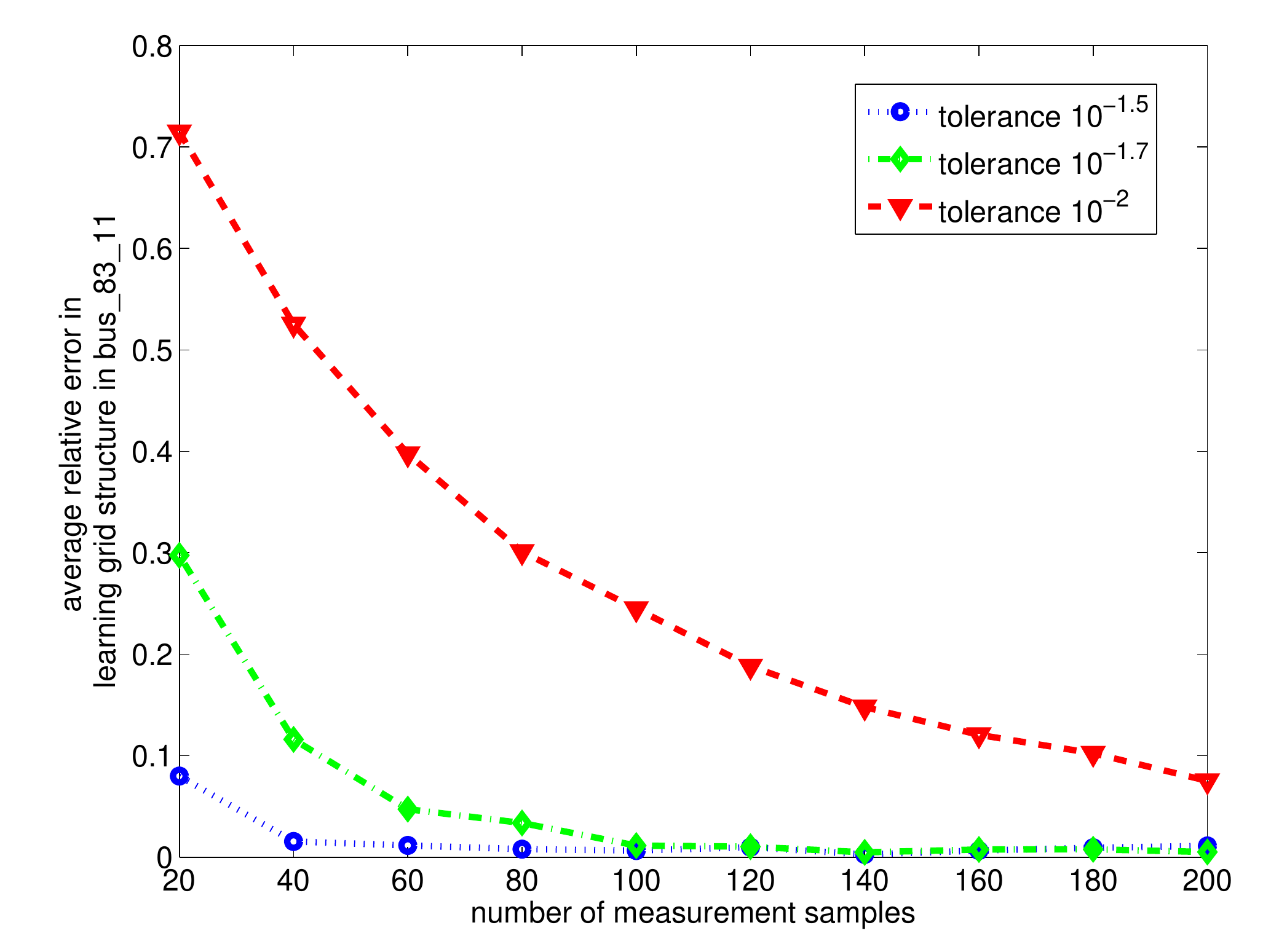}\label{fig:algo2case3}}
\vspace{-.25cm}
\caption{Average fractional errors plotted against the number of measurement samples for the LC-PF model with different tolerance thresholds - correspondent to the Algorithm 2.(a), (b) and (c) show results for the $bus\_13\_3$, $bus\_29\_1$ and $bus\_83\_11$ models respectively.
\label{fig:algo2}}
\vspace{-3mm}
\end{figure}
\squeezeup

\begin{figure}[!bt]
\centering
\includegraphics[width=0.42\textwidth,height = .35\textwidth]{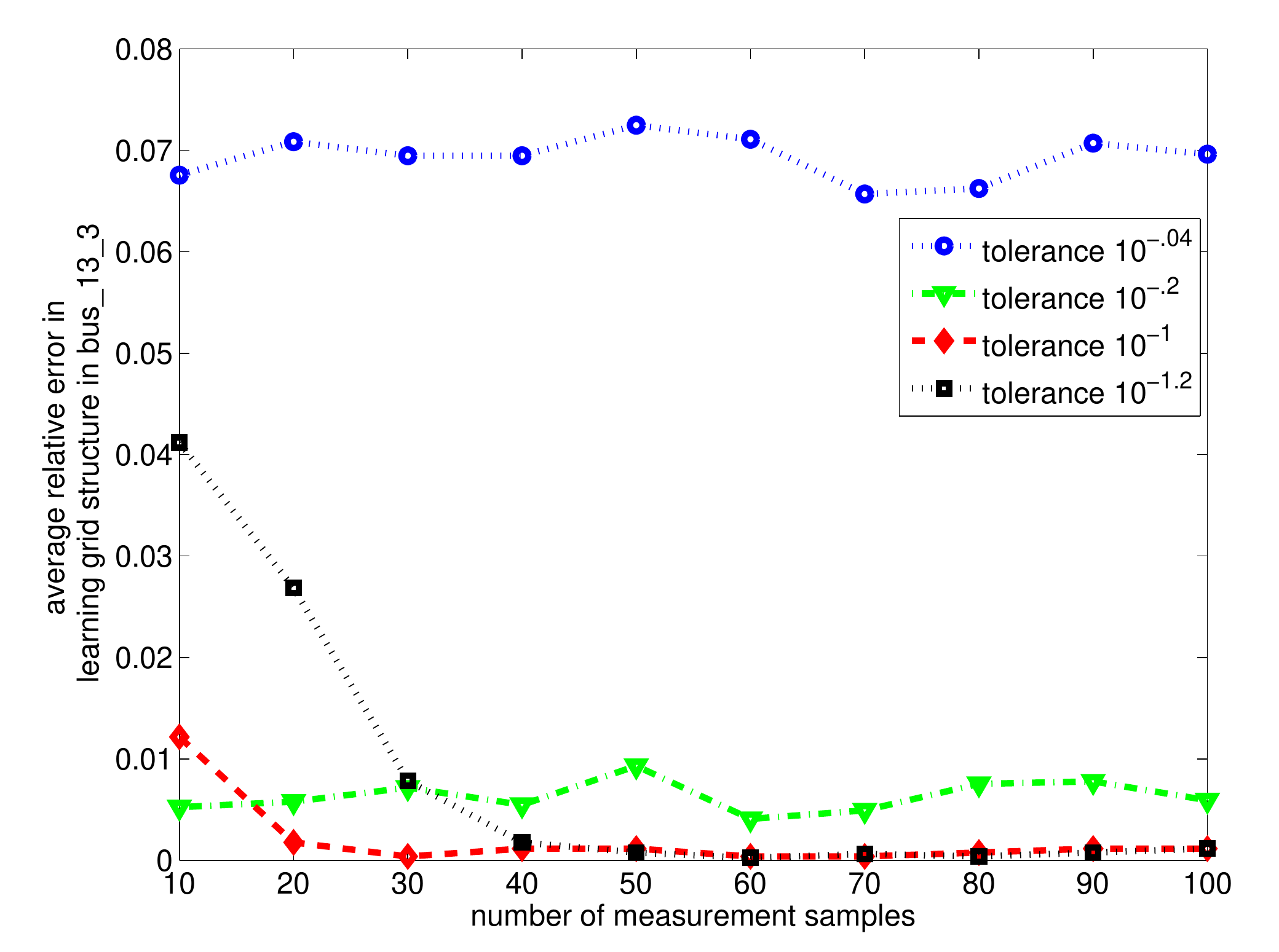}
\caption{Effect of threshold on the average fractional errors in Algorithm $1$ using the LC-PF model for the $bus\_13\_3$ (with $50$ additional non-operational edges).
\label{fig:algo2_less}}
\vspace{-3mm}
\end{figure}

\section{Conclusions \& Path Forward}
\label{sec:conclusion}
Accurate structural estimation of distribution grids is important to many applications including failure identification, power flow optimization and estimation of state variables. In this manuscript, we have developed algorithms to learn the structure of a radial distribution grid using observed nodal voltage magnitude measurements. We have used a Linear Coupled (LC) approximation that relates complex nodal voltages to the complex power consumptions to prove that second moments of nodal voltage magnitudes in radial distribution grids follow certain statistical structure/ordering. Our algorithm relies on these results to reconstruct the operational tree in a bottom up fashion -- starting from the leaves and progressing to the root of the grid.

The primary benefits of our approach are two-fold. First, our model is practical as voltage measurements are easily available at distribution grid nodes and individual devices. Second, the only assumption used regarding the statistics of the loads is the positivity of non-central correlation of nodal load profiles, which is natural for most distribution grids and more general than other assumptions discussed in the literature. We tested our algorithms on sample distribution grids and observed an exponential decay of average errors with increasing number of measurement samples. In the next part, we extend the work in this paper along the following directions: coupling structure learning with estimation of load statistics in the grid or estimation of line parameters in the grid, and learning grid structure even when measurement data is missing. Developing general algorithms for learning distribution grid operational forests that are not restricted to linearized power flow models remains a potential future direction of research.

\appendix
\section{Models of Power Flows}
\label{app:PF}
\subsection{Basic Power Flows}

Note that the operational `base-constrained spanning forest' $\cal F$ can be thought of as a spanning tree over an extended graph, where an (artificial) super node is introduced and connected with (artificial) lines to all the sub-station nodes. This trick allows, without a loss of generality, to limit our discussion in the following to spanning trees, thus replacing ${\cal F}$ by ${\cal T}$.

Let $z=(z_{ab}=r_{ab}+i x_{ab}|(ab)\in {\cal E})$ as the vector of complex line impedances in the grid. ($i^2=-1$.) Expressed in terms of the complex powers and potentials (voltages and phases) the Kirchoff laws over the operational (spanning tree) configuration ${\cal T}$ become
\begin{align}
\forall a\in{\cal V}: & P_a =p_a+i q_a=\underset{b:(ab)\in{\cal E}^{{\cal T}}}{\sum}\frac{v_a^2-v_a v_b\exp(i\theta_a-i\theta_b)}{z_{ab}^*}\label{P-complex}
\end{align}
where the real valued scalars, $v_a$ and $\theta_a$, characterize voltage magnitude and phase respectively at node $a$. 
We assume that the power within the system considered is balanced through a slack bus, $a=0$.
The set of Eqs.~(\ref{P-complex}), expressing potentials via complex powers injected at the nodes of the power graph, are called \emph{Power Flow} (PF) equations.

\subsection{Linear Coupled (LC) Approximation of Power Flows}
We linearize the PF Eqs.~(\ref{P-complex}) in the first order jointly over phase difference and voltage deviations ($v_a -1=\varepsilon_a$) from nominal, i.e. the two corrections are considered on equal footing. We arrive at the following set of equations:
\begin{align}
p_a&=\underset{b:(ab)\in{\cal E}^{{\cal T}}}{\sum}\left(\beta_{ab}(\theta_a-\theta_b)+
g_{ab}(\varepsilon_a-\varepsilon_b)\right),
\label{PF_LPV_p}\\
q_a&=\underset{b:(ab)\in{\cal E}^{{\cal T}}}{\sum}
\left(-g_{ab}(\theta_a-\theta_b)+
\beta_{ab}(\varepsilon_a-\varepsilon_b)\right)\label{PF_LPV_q}\\
\text{where~} &\forall a\in{\cal V},\forall (ab)\in{\cal E}^{{\cal T}}:~|\varepsilon_a|\ll 1, |\theta_a-\theta_b| \ll 1,\label{varepsilon_small}\\
&~~~~~~~~~~~~~~ g_{ab}\doteq\frac{r_{ab}}{x_{ab}^2+r_{ab}^2}, \beta_{ab}\doteq\frac{x_{ab}}{x_{ab}^2+r_{ab}^2} \label{g}
\end{align}

Eqs.~(\ref{PF_LPV_p},\ref{PF_LPV_q}) show coupling between phases and voltages, thus calling it the Linear-Coupled (LC) approximation, is proper. The linearity of LC-PF is used in the paper in deriving results to learn the grid structure. 
Two comments are in order. First of all, notice that the LC-PF approximation does not make any assumption about the relative strength of inductance and resistance, thus making it applicable to power distribution systems where the two line characteristics are typically of the same order. Second, expressions under the sum on the rhs of Eqs.~(\ref{PF_LPV_p},\ref{PF_LPV_q}) represent active and reactive power flows which are antisymmetric ($p_{a\to b}=-p_{b\to a}$, $q_{a\to b}=-q_{b\to a}$). This emphasizes an important consequence of linear approximation in the LC-PF model -- both active and reactive losses in lines are ignored as such losses occur at second order.

\subsection{DC-resistive approximation} 
In low-voltage distribution grids line inductances may be much smaller in magnitude than line resistances. Then the LC-PF model can be simplified even further. Indeed, taking this case to the extreme where inductance can be ignored in comparison with the resistance, $\forall \{a,b\}:\quad \beta_{ab}\ll g_{ab}$, we arrive at the following resistance-dominating version of Eqs.~(\ref{PF_LPV_p},\ref{PF_LPV_q})
\begin{align}
\forall a\in {\cal V}: p_a\approx\sum_{b:(ab)\in{\cal E}^{{\cal T}}}
g_{ab}(\varepsilon_a-\varepsilon_b),
q_a\approx\sum_{b:(ab)\in{\cal E}^{{\cal T}}}g_{ab}(\theta_b-\theta_a)\nonumber
\end{align}
We will coin this approximation DC-resistive PF. Its formulation is similar to the traditional DC flow model \cite{abur2004power} where active power flows are related to phase angles. The traditional DC model, used primarily for transmission networks, requires line inductances to dominate resistances, which is seldom observed in distribution grids.

\subsection{From Power Flows to DistFlow and LinDistFlow}

The DistFlow Eqs. introduced by Baran and Wu in \cite{89BWa,89BWb} are derived from the PF Eq.~(\ref{P-complex})
\begin{align}
& p_{a\to b}-r_{ab}\frac{p_{a\to b}^2+q_{a\to b}^2}{v_a^2}=p_b+\sum_{(bc)\in{\cal E}^{{\cal T}};c\neq a} p_{b\to c},
\label{DF_p}\\
& q_{a\to b}-x_{ab}\frac{p_{a\to b}^2+q_{a\to b}^2}{v_a^2}=q_b+\sum_{(bc)\in{\cal E}^{{\cal T}};c\neq a} q_{b\to c},
\label{DF_q}\\
& v_b^2=v_a^2-2\left(r_{ab}p_{a\to b}+x_{ab} q_{a\to b}\right)+\left(r_{ab}^2+x_{ab}^2\right)\frac{p_{a\to b}^2+q_{a\to b}^2}{v_a^2}
\label{DF_v}
\end{align}
where each of the three equations above are stated in terms of active, $p_{a\to b}$, and reactive, $q_{a\to b}$, powers over directed lines, $a\to b$, and voltages over nodes, $a\in{\cal V}$. If power losses at any line segment is negligible, Eqs.~(\ref{DF_p},\ref{DF_q},\ref{DF_v}) reduce to \cite{baran1989network}.
\begin{align}
&p_{a\to b}\approx p_b+\sum_{\substack{(bc)\in{\cal E}^{{\cal T}}\\c \neq a}} p_{b\to c},
q_{a\to b}\approx q_b+\sum_{\substack{(bc)\in{\cal E}^{{\cal T}}\\c \neq a}} q_{b\to c},
\label{lin_DF_q}\\
&\varphi_b\approx \varphi_a-2\left(r_{ab}p_{a\to b}+x_{ab} q_{a\to b}\right), \quad \varphi_a\equiv v_a^2
\label{lin_DF_v}
\end{align}
On a general graph with loops, the number of these (directed) edge-related variables in Eq.~(\ref{P-complex}) is larger then the number of phases and voltages. However, the former becomes equal to the later if the grid is a tree, where the number of edges is equal to the number of nodes minus one. Therefore when the phase and voltage at one special node (each substation node per tree in our case) is fixed at $v_0$ (say), phases and voltages at all other nodes can be reconstructed from the directional line flows, and vice versa.

Note that additional boundary conditions arise in a tree due to the requirement of active and reactive power flowing from/into any leaf being equal to its nodal injection. Thus,
\begin{align}
& \forall a\in{\cal V}_0:\quad v_a = v_0,\label{v0}\\
& \forall b\in V_l^{{\cal T}},\ \& \ (a\to b)\in{\cal E}^{{\cal T}}: p_{a\to b}=p_b,\quad q_{a\to b}=q_b,
\label{leafs}
\end{align}
where ${\cal V}_0\subset {\cal V}$ is the set of $K$ substations nodes (colored in red in Fig.~\ref{fig:city}), and $V_l^{{\cal T}}$ is the set of leaf-nodes in tree ${\cal T}$. Based on the assumption that voltage drop across any line segment is sufficiently small, we linearize Eq. (\ref{lin_DF_v}) by substituting $v_a=1+\varepsilon_a$, with $|\varepsilon_a|\ll 1$ and have
\begin{equation}
\forall (ab)\in{\cal E}^{{\cal T}}:\quad \varepsilon_a-\varepsilon_b\approx \left(r_{ab}p_{a\to b}+x_{ab} q_{a\to b}\right).
\label{lin_DF_epsilon}
\end{equation}
Notice that the LinDistFlow Eqs.~(\ref{lin_DF_q},\ref{lin_DF_epsilon}), derived assuming that the grid graph is a tree, are exactly equivalent to Eqs.~(\ref{PF_LPV_p},\ref{PF_LPV_q}) of the Linear Coupled (LC) Approximation.

\section*{Acknowledgment}

The work at LANL was funded by the Advanced Grid Modeling Program in the Office of Electricity in the US Department of Energy and was carried out under the auspices of the National Nuclear Security Administration of the U.S. Department of Energy at Los Alamos National Laboratory under Contract No. DE-AC52-06NA25396.

\bibliographystyle{IEEETran}
\bibliography{../../Bib/FIDVR,../../Bib/SmartGrid,../../Bib/voltage,../../Bib/trees}

\begin{IEEEbiography}[{\includegraphics[width=1in,height=2in,clip,keepaspectratio]{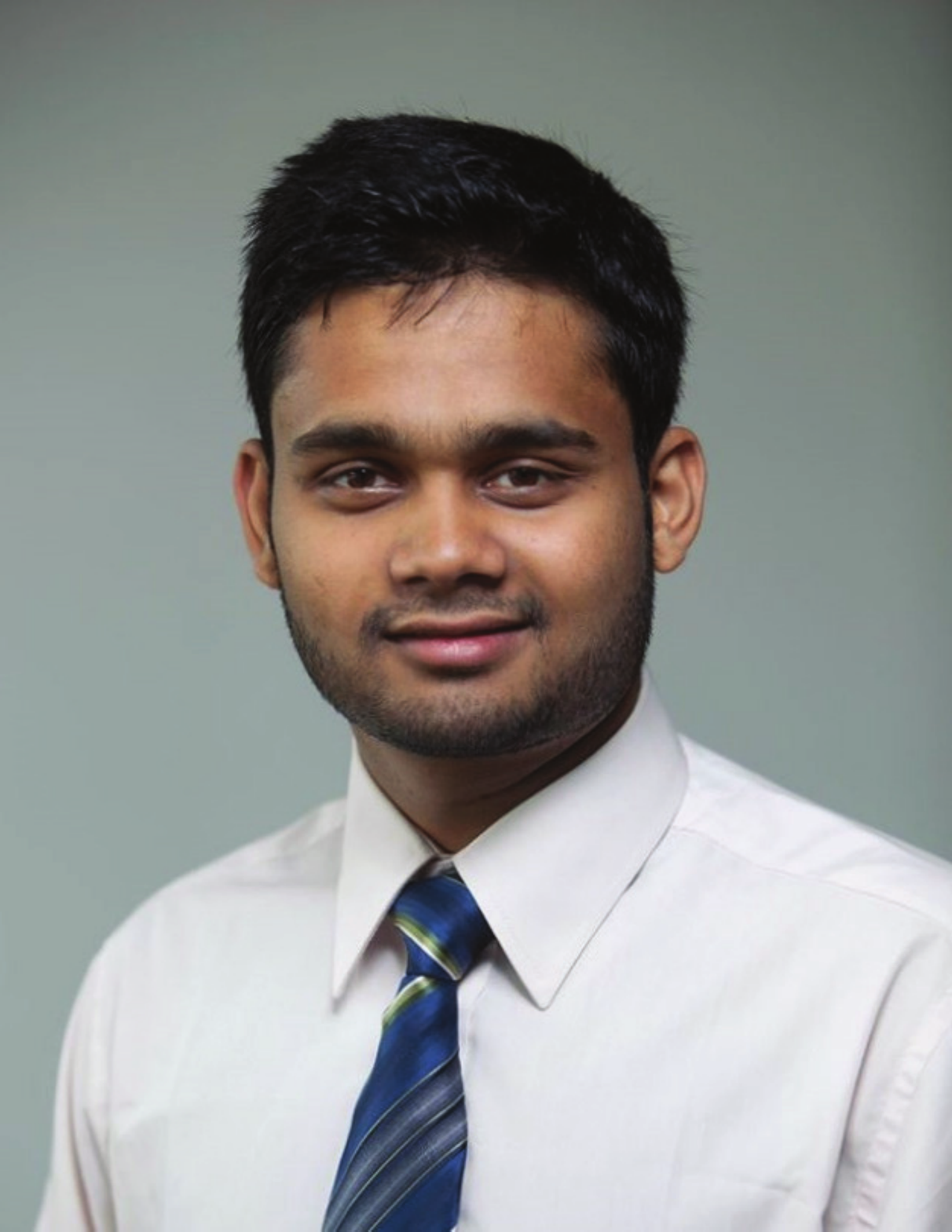}}]{Deepjyoti Deka}
Deepjyoti Deka received his M.S. in Electrical Engineering from University of Texas, Austin in 2011, and his B.Tech in Electronics and Communication Engineering from IIT Guwahati, India, in 2009 for which he was awarded the Institute Silver Medal. He is currently a PhD candidate in Electrical Engineering at UT Austin. His research focusses on the design and analysis of power grid structure, operations and data security. He is also interested in modeling and optimization in social and physical networks. He has held internship positions at Los Alamos National Lab, Los Alamos NM, Electric Reliability Council of Texas, Taylor TX, and Qualcomm Inc, San Diego CA.
\end{IEEEbiography}

\vspace{-1.5cm}
\begin{IEEEbiography}[{\includegraphics[width=1in,height=2in,clip,keepaspectratio]{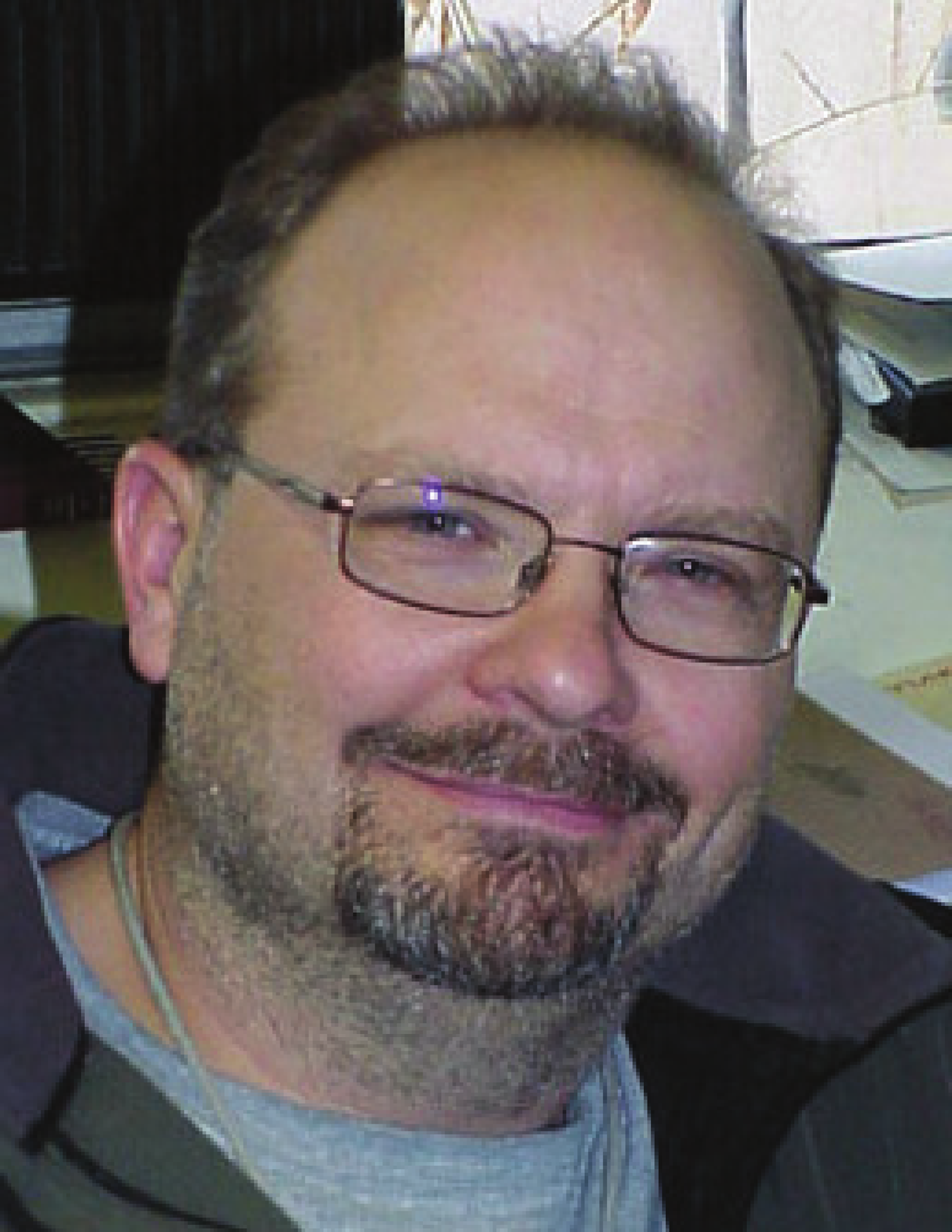}}]{ScottBackhaus}
Scott Backhaus received the Ph.D. degree in physics from the University
of California at Berkeley in 1997 in the area of experimental macroscopic
quantum behavior of superfluid He-3 and He-4.
In 1998, he came to Los Alamos, NM, was Director's Funded
Postdoctoral Researcher from 1998 to 2000, a Reines Postdoctoral
Fellow from 2001 to 2003, and a Technical Staff Member from 2003 to
the present. While at Los Alamos, he has performed both experimental
and theoretical research in the area of thermoacoustic energy conversion
for which he received an R\&D 100 award in 1999 and Technology
Review's Top 100 Innovators Under 35 [award in 2003]. Recently, his
attention has shifted to other energy-related topics including the fundamental
science of geologic carbon sequestration and grid-integration of
renewable generation.
\end{IEEEbiography}

\vspace{-1.5cm}
\begin{IEEEbiography}[{\includegraphics[width=1in,height=1.25in,clip,keepaspectratio]{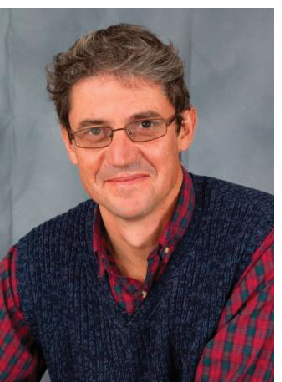}}]{Michael Chertkov}
Dr. Chertkov's areas of interest include statistical and
mathematical physics applied to energy and communication networks,
machine learning, control theory, information theory, computer science,
fluid mechanics and optics. Dr. Chertkov received his Ph.D. in
physics from the Weizmann Institute of Science in 1996, and his
M.Sc. in physics from Novosibirsk State University in 1990.
After his Ph.D., Dr. Chertkov spent three years at Princeton
University as a R.H. Dicke Fellow in the Department of Physics.
He joined Los Alamos National Lab in 1999, initially as a J.R.
Oppenheimer Fellow in the Theoretical Division. He is now a
technical staff member in the same division. Dr. Chertkov has
published more than 130 papers in these research
areas. He is an editor of the Journal
of Statistical Mechanics (JSTAT), associate editor of IEEE Transactions on
Control of Network Systems, a fellow of the American Physical
Society (APS), and a Founding Faculty Fellow of Skoltech (Moscow, Russia).
\end{IEEEbiography}

\end{document}